\documentclass[12pt]{amsart}
\usepackage{a4wide,enumerate,color}
\usepackage[normalem]{ulem} 
\usepackage{amsmath}
\usepackage{scrextend} 
\allowdisplaybreaks
\usepackage{hyperref}

\let\pa\partial  
\let\na\nabla  
  
\newcommand{\N}{{\mathbb N}}  
\newcommand{\R}{{\mathbb R}} 
\newcommand{\diver}{\operatorname{div}}  
\newcommand{\dom}{{\mathcal O}}

\newcommand{\F}{{\mathcal F}}
\newcommand{\E}{{\mathbb E}}
\renewcommand{\L}{{\mathcal L}}
\newcommand{\Prob}{\mathbb{P}}
\renewcommand{\d}{\textnormal{d}}

\newtheorem{theorem}{Theorem}   
\newtheorem{lemma}[theorem]{Lemma}

\newtheorem{remark}[theorem]{Remark}   
\newtheorem{corollary}[theorem]{Corollary}

\begin{document}  

\title[Stochastic Shigesada--Kawasaki--Teramoto model]{Global martingale solutions for 
a stochastic Shigesada--Kawasaki--Teramoto population model} 

\author[G. Dhariwal]{Gaurav Dhariwal}
\address{Institute for Analysis and Scientific Computing, Vienna University of  
	Technology, Wiedner Hauptstra\ss e 8--10, 1040 Wien, Austria}
\email{gaurav.dhariwal@tuwien.ac.at} 

\author[F. Huber]{Florian Huber}
\address{Institute for Analysis and Scientific Computing, Vienna University of  
	Technology, Wiedner Hauptstra\ss e 8--10, 1040 Wien, Austria}
\email{florian.huber@asc.tuwien.ac.at}

\author[A. J\"ungel]{Ansgar J\"ungel}
\address{Institute for Analysis and Scientific Computing, Vienna University of  
	Technology, Wiedner Hauptstra\ss e 8--10, 1040 Wien, Austria}
\email{juengel@tuwien.ac.at} 

\date{\today}

\thanks{All authors acknowledge partial support from   
the Austrian Science Fund (FWF), grants I3401, P30000, P33010, W1245, and F65.} 

\begin{abstract}
The existence of global nonnegative martingale solutions to a cross-diffusion system of
Shigesada--Kawasaki--Teramoto type with multiplicative noise is proven. 
The model describes the segregation dynamics of populations with an
arbitrary number of species. The diffusion matrix is generally 
neither symmetric nor positive semidefinite, which excludes standard methods.
Instead, the existence proof is based on the entropy structure of the model,
approximated by a Wong--Zakai argument, and on suitable higher moment estimates
and fractional time regularity.
In the case without self-diffusion, the lack of regularity is overcome by
carefully exploiting the entropy production terms. 
\end{abstract}

 

\maketitle


\section{Introduction}

Shigesada, Kawasaki, and Teramoto (SKT) suggested in their seminal paper \cite{SKT79}
a deterministic cross-diffusion system for two competing species, 
which is able to describe the segregation of the populations. A random
influence of the environment or the lack of knowledge of certain biological parameters
motivate the introduction of noise terms, leading to the system
\begin{equation}\label{1.skt}
\begin{aligned}
  & \d u_1 - \Delta\big(a_{10}u_1 + a_{11}u_1^2 + a_{12}u_1u_2\big)\d t
  = \sigma_1(u_1)\d W_1(t), \\
  & \d u_2 - \Delta\big(a_{20}u_2 + a_{21}u_1u_2 + a_{22}u_2^2\big)\d t
  = \sigma_2(u_2)\d W_2(t)\quad\mbox{in }\dom,\ t>0,
\end{aligned}
\end{equation}
where $u_i=u_i(\omega,x,t)$ describes the density of the $i$th species ($i=1,2$),
$\omega\in\Omega$ is the stochastic variable, $x\in\dom$
is the spatial variable, and $t\ge 0$ is the time, $a_{ij}\ge 0$ are
some parameters, $(W_1,W_2)$ is a two-dimensional Wiener process,
and $\dom\subset\R^d$ ($d\ge 1$) is a bounded domain. An admissible example
of the stochastic diffusion term is 
\begin{equation}\label{1.sigma}
  \sigma_i(u_i) = \frac{u_i}{1+u_i^{1-\gamma}}, \quad\mbox{where }
	0<\gamma\le 1,\ i=1,2.
\end{equation}
Details on the stochastic framework will be given in Section \ref{sec.main}.
The equations are supplemented with 
initial and no-flux boundary conditions (see \eqref{1.bic} below).
The original system in \cite{SKT79} also contains a deterministic environmental
potential, which are neglected here for simplicity.

The key difficulty of system \eqref{1.skt} is the fact that the diffusion matrix
associated to \eqref{1.skt} is generally neither symmetric nor positive semidefinite. 
In particular, standard semigroup theory is not applicable. These issues
have been overcome in \cite{ChJu04,ChJu06} in the deterministic case
by revealing a formal gradient-flow or entropy structure. 
The task is to extend this idea to the stochastic setting.

The aim of this work is to prove the existence of global nonnegative
martingale solutions to system \eqref{1.skt}. The paper is a continuation 
of our previous works \cite{DJZ19,DHJKN20}. The work \cite{DJZ19} was concerned
with a SKT-type system, in which the coefficients of the associated diffusion
matrix depend quadratically on the densities (and not linearly as in \eqref{1.skt}). 
This allowed us to work in a Hilbert space framework,
leading to a novel result through a standard approach. The paper \cite{DHJKN20}
exploited the entropy structure of a general class
of cross-diffusion systems with volume filling, leading to bounded martingale
solutions. The $L^\infty$ bound follows from the entropy structure
and implies higher moment estimates, which were further used
to establish the tightness of laws. Unfortunately, this idea does not work for
system \eqref{1.skt}, since the entropy structure is different and $L^\infty$
bounds cannot be expected.
Therefore, we need to develop new estimates to overcome this issue.

In fact, we prove the existence of martingale solutions to a
SKT-type system involving an arbitrary number of species. We consider
\begin{equation}\label{1.eq}
  \d u_i - \Delta\bigg(a_{i0}u_i + \sum_{j=1}^n a_{ij}u_iu_j\bigg)\d t
	= \sum_{j=1}^n \sigma_{ij}(u)\d W_j(t)\quad\mbox{in }\dom,\ t>0,\ i=1,\ldots,n,
\end{equation}
with the initial and no-flux boundary conditions
\begin{equation}\label{1.bic}
  u_i(0)=u_i^0\quad\mbox{in }\dom,\quad
	\na\bigg(a_{i0}u_i + \sum_{j=1}^n a_{ij}u_i u_j\bigg)\cdot\nu = 0
	\quad\mbox{on }\pa\dom,\ t>0,\ i=1,\ldots,n.
\end{equation}
Here, $a_{ij}\ge 0$ for $i=1,\ldots,n$ and $j=0,\ldots,n$, $u=(u_1,\ldots,u_n)$
is the vector of population densities, 
$(W_1,\ldots,W_n)$ is an $n$-dimensional Wiener process,
$\nu$ is the exterior unit normal vector to $\pa\dom$, and $u_i^0$ is a possibly
random initial datum. We call $a_{i0}$ the diffusion coefficients, $a_{ii}$
the self-diffusion coefficients, and $a_{ij}$ for $i\neq j$ the cross-diffusion
coefficients. 
We say that system \eqref{1.eq}-\eqref{1.bic} is {\em with self-diffusion} if
$a_{i0}\ge 0$, $a_{ii} > 0$ for all $i=1,\ldots,n$, and it is
{\em without self-diffusion} if $a_{i0}>0$, $a_{ii}=0$ for all $i=1,\ldots,n$.

The deterministic analog of \eqref{1.eq} was formally derived from a random-walk
lattice model in \cite{ZaJu17} and rigorously derived from a nonlocal population
system in the triangular case in \cite{Mou20} and from interacting
particle systems in the general case in \cite{CDHJ20}.

Equations \eqref{1.eq} can be written as
$$
  \d u_i - \diver\bigg(\sum_{j=1}^n A_{ij}(u)\na u_j\bigg)\d t
	= \sum_{j=1}^n \sigma_{ij}(u)\d W_j(t)\quad\mbox{in }\dom,\ t>0,\ i=1,\ldots,n,
$$
with the diffusion matrix $A(u)=(A_{ij}(u))$, where
\begin{equation}\label{1.A}
  A_{ij}(u) = \delta_{ij}\bigg(a_{i0} + \sum_{k=1}^na_{ik}u_k\bigg) + a_{ij}u_i,
	\quad i,j=1,\ldots,n.
\end{equation}
As mentioned above, the main difficulty that this matrix is generally neither 
symmetric nor positive semidefinite is overcome by exploiting the entropy structure.
This means that there exists a function $h:[0,\infty)^n\to[0,\infty)$, 
called an entropy density, such that the deterministic
analog of \eqref{1.eq} can be written in terms of the entropy variables
(or chemical potentials) $w_i=\pa h/\pa u_i$ as
\begin{equation}\label{1.w}
  \pa_t u_i(w) - \diver\bigg(\sum_{j=1}^n B_{ij}(w)\na w_j\bigg) = 0, \quad
	i=1\ldots,n,
\end{equation}
where $w=(w_1,\ldots,w_n)$, $u_i$ depends on $w$, and $B(w)=A(u(w))h''(u(w))^{-1}$ 
with $B=(B_{ij})$ turns out to
be positive semidefinite. For the deterministic analog of \eqref{1.eq}, it was 
shown in \cite{CDJ18} that the entropy density is given by
\begin{equation}\label{1.h}
  h(u) = \sum_{i=1}^n\pi_i \big(u_i(\log u_i-1)+1\big),
\end{equation}
where the numbers $\pi_i>0$ are assumed to satisfy 
$\pi_ia_{ij}=\pi_ja_{ji}$ for all $i,j=1,\ldots,n$.
This condition is the detailed-balance condition for the Markov chain associated
to $(a_{ij})$, and $(\pi_1,\ldots,\pi_n)$ is the corresponding reversible stationary
measure. Using $w_i$ in \eqref{1.w} as a test function and summing over $i=1,\ldots,n$,
a formal computation shows that
\begin{equation}\label{1.ei}
  \frac{\d}{\d t}\int_\dom h(u)\d x
	+ 2\int_\dom\sum_{i=1}^n\pi_i\bigg(2a_{i0}|\na\sqrt{u_i}|^2
	+ a_{ii}|\na u_i|^2 + \sum_{j\neq i}a_{ij}|\na\sqrt{u_iu_j}|^2\bigg)\d x = 0.
\end{equation}
A similar expression holds in the stochastic setting; see Lemma \ref{lem.ei}.
It provides gradient estimates for $\sqrt{u_i}$ if $a_{i0}>0$ and for
$u_i$ if $a_{ii}>0$. Moreover, having proved the existence of a solution $w$
to an approximate version of \eqref{1.eq} leads to the positivity
of $u_i(w)=\exp(w_i/\pi_i)$ (and nonnegativity after passing to the 
de-regularization limit).

In the stochastic setting, we face some technical obstacles due to
It\^o's lemma and the treatment of the multiplicative noise.
Our idea, first used in \cite{DHJKN20}, is to replace the Wiener process
by a Wong--Zakai approximation and to discretize the equations by a stochastic
Galerkin method. We apply a variant of the boundedness-by-entropy method \cite{Jue15}
to the resulting system of differential equations, providing the positivity of 
the approximate population densities and a priori estimates uniform 
in the Galerkin dimension.
The limit of vanishing Wong--Zakai parameter requires the existence of solutions
to another Galerkin approximation, leading to strong solutions up to a stopping time.
The tightness of the laws of the approximate solutions follows from
the uniform estimates, and the Skorokhod--Jakubowski theorem implies the 
pointwise convergence of the sequence of approximate solutions.

The difference in the strategy of our proof, compared to \cite{DHJKN20}, becomes
apparent in the final steps. First, we need higher moment estimates, which followed
in \cite{DHJKN20} from the $L^\infty(\dom)$ bound, but here we have less regularity.
Second, we are proving a new fractional time regularity result, needed
to show the tightness of the laws in $L^2(\dom)$. 
A further difference to \cite{DHJKN20} comes from the low regularity of the
gradients when self-diffusion vanishes. Indeed, the regularity for $\na\sqrt{u_i}$
in $L^2$ from \eqref{1.ei} does not allow us to define products like $u_iu_j$. 
The idea is to
exploit the Laplace structure in \eqref{1.eq} and the $L^2$ regularity for
$\na\sqrt{u_iu_j}$ coming from the entropy estimate \eqref{1.ei}.

As a corollary of our main results stated in Section \ref{sec.main}, we obtain the
existence of a global martingale solution to \eqref{1.skt} with the particular
stochastic noise term \eqref{1.sigma}. 

\begin{corollary}[Martingale solutions to the SKT model]
Let $\dom\subset\R^3$ be a bounded domain with smooth boundary
and let either $a_{11}>0$, $a_{22}>0$, $\gamma\le 1$ or 
$a_{10}>0$, $a_{20}>0$, $\gamma\le 2/5$.
Then there exists a global nonnegative martingale solution to 
\eqref{1.skt}, \eqref{1.sigma}, and \eqref{1.bic}.
\end{corollary}

Deterministic cross-diffusion systems of SKT type with two species 
have been intensively studied in the literature. 
First existence results were proven under restrictive conditions
on the parameters, for instance in one space dimension \cite{Kim84},
for the triangular system with $a_{21}=0$ \cite{LNW98}, or for small
cross-diffusion parameters, since in the latter situation the diffusion matrix
is positive definite \cite{Deu87}. Amann \cite{Ama89} proved that a priori 
estimates in the $W^{1,p}(\dom)$ norm with $p>d$ are sufficient to prove
the global existence of solutions to quasilinear parabolic systems, and
he applied this result to the triangular SKT system. The first global existence
proof without any restriction on the parameters $a_{ij}$ (except nonnegativity)
was achieved in \cite{GGJ03} in one space dimension. This result was generalized
to several space dimensions in \cite{ChJu04,ChJu06} and to the whole space
problem in \cite{Dre08}. SKT-type systems with nonlinear coefficients $A_{ij}(u)$, 
but still for two species, were analyzed in \cite{DLM14,DLMT15}.
Global existence results for SKT-type models with an arbitrary number of species
and under a detailed-balance condition were first proved in \cite{CDJ18} and
later generalized in \cite{LeMo17}. 

There are only very few results for stochastic SKT-type systems. 
The first result for global martingale solutions needed
quadratic diffusion coefficients, since
this allows one to work in a Hilbert space framework \cite{DJZ19}. 
A stronger solution concept was used in \cite{KuNe20},
leading to local-in-time pathwise mild solutions, but only for positive definite
diffusion matrices $A(u)$.

The paper is organized as follows. The stochastic framework and main results
are given in Section \ref{sec.main}. In Section \ref{sec.approx}, the 
existence of approximate solutions is proved and uniform bounds are derived 
from the entropy estimate. The existence of global martingale solutions
is proved in Section \ref{sec.ex} in the case with self-diffusion and 
in Section \ref{sec.no} in the case without self-diffusion.
Estimates for the deterministic SKT system without self-diffusion, which are
needed for the approximate stochastic problem, are derived in Appendix
\ref{app}. As a by-product, we obtain an existence result for the deterministic SKT
system without self-diffusion with a simpler proof and for a more general 
situation compared to \cite{ChJu06}.


\section{Notation and main result}\label{sec.main}

\subsection{Notation and stochastic framework}

Let $\dom\in\R^d$ be a bounded domain. The Lebesgue and
Sobolev spaces are denoted by $L^p(\dom)$ and $W^{k,p}(\dom)$, respectively,
where $p\in[1,\infty]$, $k\in\N$, and we set $H^k(\dom)=W^{2,k}(\dom)$.
We write $\|u\|_{L^2(\dom)}^2=\sum_{i=1}^n\|u_i\|_{L^2(\dom;\R^n)}^2$ for functions
$u=(u_1,\ldots,u_n)\in L^2(\dom)$ and use this notation in related situations.
We write $\langle\cdot,\cdot\rangle$ for the dual product between a Banach space and
its dual. If $u\in L^p(\dom)$, $v\in L^q(\dom)$ with $1/p+1/q=1$ and $p>1$, 
we have $\langle u,v\rangle=\int_\dom uv\d x$.

Let $(\Omega,\F,\Prob)$ be a probability space endowed with a complete
right-continuous filtration $\mathbb{F}=(\F_t)_{t\ge 0}$. The space
$L^2(\Omega;H)$ for a Hilbert space $H$ consists of all $H$-valued random
variables $u$ such that $\E\|u\|_H^2=\int_\Omega\|u(\omega)\|_H^2\Prob(\d\omega)
<\infty$. Let $(\eta_k)_{k=1,\ldots,n}$ be the canonical basis of $\R^n$.
The space of Hilbert--Schmidt operators from $\R^n$ to $L^2(\dom)$ is defined by
$$
  \L_2(\R^n;L^2(\dom)) = \bigg\{L:\R^n\to L^2(\dom)\mbox{ linear continuous}:
	\sum_{k=1}^n\|L\eta_k\|_{L^2(\dom)}^2<\infty\bigg\}
$$
and endowed with
the norm $\|L\|_{\L_2(\R^n;L^2(\dom))}^2=\sum_{k=1}^n\|L\eta_k\|_{L^2(\dom)}^2$.
The stochastic diffusion $\sigma=(\sigma_{ij}):
\R^n\to\R^{n\times n}$ is assumed to be
${\mathcal B}(L^2(\dom;\R^n))/\mathcal{B}(\mathcal{L}_2(\R^n;L^2(\dom)))$-measurable 
and $\mathbb{F}$-adapted.

\subsection{Assumptions and main result}

We impose the following assumptions:

\begin{labeling}{(A44)}
\item[(A1)] Domain: $\dom\subset\R^d$ ($d\ge 1$) is a bounded domain with
$\pa\dom\in C^{\ell}$ and $\ell\in\N$ satisfies $\ell>d/2+2$. Let $T>0$ and
set $Q_T=\dom\times(0,T)$.

\item[(A2)] Initial datum: $u^0\in L^{r}(\Omega;L^2(\dom;\R^n))$ 
for $r>(2/d)\max\{8,d+2\}$ is $\F_0$-measurable and 
$u_i^0\ge 0$ for a.e.\ $x\in\dom$ $\Prob$-a.s., $i=1,\ldots,n$.

\item[(A3)] Diffusion matrix: $a_{ij}\ge 0$ for $i=1,\ldots,n$, $j=0,\ldots,n$
and the detailed-balance condition is satisfied, 
i.e., there exist numbers $\pi_1,\ldots,\pi_n>0$ such that
\begin{equation}\label{2.dbc}
  \pi_i a_{ij} = \pi_j a_{ji} \quad\mbox{for all }i,j=1,\ldots,n.
\end{equation}

\item[(A4)] Multiplicative noise: $\sigma:L^2(\dom;\R^n)\to\L_2(\R^n;L^2(\dom;\R^n))$ 
satisfies $\sigma_{ij}(u)=0$ for all $u\in[0,\infty)$ with $u_i=0$,
and there is a constant $C_\sigma>0$ such that for any $u$, $v\in L^2(\dom)$,
\begin{align*}
  \|\sigma(u)-\sigma(v)\|_{\L_2(\R^n;L^2(\dom))}
	&\le C_\sigma\|u-v\|_{L^2(\dom)}, \\
	\|\sigma(u)\|_{\L_2(\R^n;L^2(\dom))} 
	&\le C_\sigma\big(1+\|u\|_{L^2(\dom)}^\gamma\big),
\end{align*}
where $0<\gamma\le 1$ (with self-diffusion) or
$0<\gamma\le 2/d$ (without self-diffusion). 

\item[(A5)] Interaction of entropy density and noise: There exists $C_h>0$ such that
for all $u\in(0,\infty)^n$,
\begin{align*}
  \max_{j=1,\ldots,n}\bigg|\sum_{i=1}^n\sigma_{ij}(u) \log u_i\bigg|^2
	&+ \bigg|\sum_{i,j,k=1}^n\frac{\pa\sigma_{ij}}{\pa u_k}(u)
	\sigma_{jk}(u)\log u_i\bigg| \\
	&{}+ \bigg|\sum_{i,j=1}^n\frac{\sigma_{ij}(u)^2}{u_j}\bigg| 
	\le C_h\bigg(1 + \sum_{i=1}^n u_i(\log u_i-1)\bigg).
\end{align*}
Furthermore, $u\mapsto \sum_{i,j,k=1}^n(\pa\sigma_{ij}/\pa u_k(u)\sigma_{jk}(u)$
is assumed to be Lipschitz continuous.
\end{labeling}

Let us discuss these assumptions.
The boundary regularity in Assumption (A1) is used to define a Galerkin space
embedded into $W^{2,\infty}(\dom)$, thus avoiding issues with the regularity
of the diffusion coefficients.
Assumption (A2) on the initial datum can be relaxed, since we only need
the integrability of $u_i^0\log u_i^0$. The detailed-balance condition in
Assumption (A3) is needed to derive the entropy inequality, which provides a priori
estimates (see Lemma \ref{lem.ei}). We may replace this assumption by
$$
  a_{ii} > \frac14\sum_{j=1}^n\big(\sqrt{a_{ij}}-\sqrt{a_{ji}}\big)^2,\quad
	a_{ii}>0\quad
	\mbox{for }i=1,\ldots,n,
$$
which expresses that self-diffusion dominates cross-diffusion \cite[Lemma 6]{CDJ18}.
The Lipschitz continuity of the stochastic diffusion term in Assumption (A4)
in the case with self-diffusion is a standard condition for stochastic PDEs; 
see, e.g., \cite{PrRo07}. Without self-diffusion, we have less regularity and
therefore we need a sublinear condition for $\sigma$.
The condition that $\sigma_{ij}$ vanishes at $u_i=0$ 
ensures the nonnegativity of the solution. 
Assumption (A5) allows us to compensate the singularity at zero for 
$\pa h/\pa u_i=\pi_i\log u_i$ when we derive the entropy estimate.
For instance, the stochastic diffusion term
$$
  \sigma_{ij}(u) = \frac{u_i\delta_{ij}}{1+u_i^{1-\gamma}}
	\quad \mbox{for all }u\in[0,\infty)^n,\ i=1,\ldots,n,
$$
satisfies Assumptions (A4)--(A5). 
Only finite-dimensional Wiener processes instead of infinite-dimensional ones 
are considered, because the specific structure of the interaction between the entropy 
density and stochastic diffusion in Assumption (A5) becomes clearer.

\begin{theorem}[Existence, with self-diffusion]\label{thm.ex}
Let Assumptions (A1)--(A5) hold, $T>0$, and let $a_{ii}>0$ for all $i=1,\ldots,n$. 
Then there exists a global martingale
solution to \eqref{1.eq}--\eqref{1.bic} satisfying 
$\widetilde u_i(x,t)\ge 0$ a.e.\ in $Q_T$ $\widetilde\Prob$-a.s.,
$i=1,\ldots,n$. More precisely, there exists a triple 
$(\widetilde U,\widetilde W,\widetilde u)$ such that 
$\widetilde U=(\widetilde\Omega,\widetilde{\mathcal{F}},\widetilde{\Prob},
\widetilde{\mathbb{F}})$ is a stochastic basis with filtration 
$\widetilde{\mathbb{F}}=(\widetilde{\mathcal{F}}_t)_{t\in[0,T]}$, $\widetilde W$
is an $\R^n$-valued Wiener process {on this filtered probability space}, 
and $\widetilde u(t)=(\widetilde u_1(t),\ldots,\widetilde u_n(t))$ 
is a progressively measurable
stochastic process for all $t\in[0,T]$ such that for $i=1,\ldots,n$,
$$
  \widetilde u_i\in L^2(\widetilde\Omega;C^0([0,T];H^{m}(\dom)'))\cap
	L^2(\widetilde\Omega;L^2(0,T;H^1(\dom))),
$$
where $m>d/2+1$, the law of $\widetilde u_i(0)$ is the same as for $u_i^0$,
and $\widetilde u$ satisfies for all $\phi_i\in H^m(\dom)$ and $i=1,\ldots,n$,
\begin{align}
  \langle\widetilde u_i(t),\phi_i \rangle
	&= \langle\widetilde u_i(0),\phi_i\rangle
	+ \int_0^t\int_\dom\na\bigg(a_{i0}\widetilde u_i(s) 
	+ \sum_{j=1}^n a_{ij}\widetilde u_i(s)
	\widetilde u_j(s)\bigg)\cdot\na\phi_i\d x\d s \nonumber \\
	&\phantom{xx}+ \sum_{j=1}^n\int_\dom\bigg(\int_0^t\sigma_{ij}(\widetilde u(s))
	\d\widetilde W_j(s)\bigg)\phi_i\d x, \label{1.weak}
\end{align}
where $\langle\cdot,\cdot\rangle$ is the dual product between $H^m(\dom)'$ and
$H^m(\dom)$.
\end{theorem}

\begin{theorem}[Existence, without self-diffusion]\label{thm.no}
Let Assumptions (A1)--(A5) hold, $T>0$, $d\le 3$, 
and let $a_{i0}>0$, $a_{ii}=0$ for all $i=1,\ldots,n$. 
Then there exists a global martingale
solution to \eqref{1.eq}--\eqref{1.bic}, satisfying 
$\widetilde u_i(x,t)\ge 0$ a.e.\ in $Q_T$ $\widetilde\Prob$-a.s.,
$i=1,\ldots,n$. More precisely, there exists a triple
$(\widetilde U,\widetilde W,\widetilde u)$, where $\widetilde U$ and
$\widetilde W$ are as in Theorem \ref{thm.ex}, 
$\widetilde u(t)=(\widetilde u_1(t),\ldots,\widetilde u_n(t))$ 
is a progressively measurable
stochastic process for all $t\in[0,T]$ such that for $i=1,\ldots,n$, $j\neq i$,
\begin{align*}
  & \widetilde u_i\in L^2(\widetilde\Omega;C^0([0,T];H^{3}(\dom)'))\cap
	L^2(\widetilde\Omega;L^{8/7}(0,T;W^{1,{8/7}}(\dom))), \\
	& \widetilde u_i\widetilde u_j\in L^2(\widetilde\Omega;
	L^{8/7}(0,T;W^{1,{8/7}}(\dom))),
\end{align*}
the law of $\widetilde u_i(0)$ is the same as for $u_i^0$, and
$\widetilde u$ satisfies the weak formulation \eqref{1.weak} for all
$\phi_i\in H^3(\dom)$.
\end{theorem}


\section{Approximate scheme and entropy estimates}\label{sec.approx}

We prove Theorems \ref{thm.ex} and \ref{thm.no} by approximating system 
\eqref{1.eq} by a stochastic Galerkin method, using the Wong--Zakai approximation 
of the Wiener process, and deriving some entropy estimates.

\subsection{Stochastic Galerkin approximation}

The existence of a strong (in the probability sense) solution to the Galerkin
approximation up to a stopping time is proved by using the Banach fixed-point
theorem. For this, we project \eqref{1.eq} onto the finite-dimensional
Hilbert space $H_N=\operatorname{span}\{e_1,\ldots,e_N\}$, where $N\in\N$ and
$(e_j)_{j\in\N}$ is an orthonormal basis of $L^2(\dom)$ such that 
$H_N\subset W^{2,\infty}(\dom)$. 
For instance, $(e_j)$ may consist of the eigenfunctions
of $-\Delta$ on $\dom$ with homogeneous Neumann boundary conditions. At this point,
we need the regularity of $\pa\dom$ to ensure that $e_j\in H^\ell(\dom)\hookrightarrow
W^{2,\infty}(\dom)$ (which requires that $\ell>d/2+2$).
Furthermore, let $\Pi_N:L^2(\dom)\to H_N$, $\Pi_N(v)=\sum_{k=1}^N
\langle v,e_k\rangle e_k$ for $v\in L^2(\dom)$ be the projection onto $H_N$.

The approximate problem is the following system of stochastic differential equations,
\begin{equation}\label{2.sde}
  \d u_i^{(N)} = \Pi_N\Delta\bigg(a_{i0}u_i^{(N)} + \sum_{j=1}^n 
	a_{ij}u_i^{(N)}u_j^{(N)}\bigg)\d t 
 	+ \sum_{j=1}^n\Pi_N\sigma_{ij}(u^{(N)})\d W_j(t)
\end{equation}
for $i=1,\ldots,n$, 
with the initial conditions
\begin{equation}\label{2.ic}
  u_i^{(N)}(0) = \Pi_N(u_i^0), \quad i=1,\ldots,n.
\end{equation}

Given $T>0$, we introduce the space $X_T=L^2(\Omega;C^0([0,T];H_N))$ with the norm
$\|u\|_{X_T}^2=\E(\sup_{0<t<T}\|u(t)\|_{H_N})^2$. For given $R>0$ and
$u\in X_T$, we define the exit time 
$\tau_R:=\inf\{t\in[0,T]:\|u(t)\|_{L^2(\dom)}>R\}$.
Then $\{\omega \in \Omega : \tau_R(\omega) > t \}$ belongs to $\mathcal{F}_t$ 
for every $t \in [0,T]$ and $\tau_R$ is an $\mathbb{F}$-stopping time.
We define the fixed-point operator $S:X_T\to X_T$ by
\begin{align*}
  \langle S_N(u)(t),\phi\rangle
	&= \langle u^0,\phi\rangle - \int_0^t\bigg\langle 
	\bigg(a_{i0}u_i+\sum_{j=1}^n a_{ij}u_iu_j\bigg),\Delta\phi_i\bigg\rangle\d s \\
  &\phantom{xx}{}+ \int_0^t\sum_{j=1}^n\langle\sigma_{ij}(u)\d W_j(s),\phi_i\rangle,
\end{align*}
for $u\in X_T$ and $\phi_i\in H_N$
satisfying $\na\phi_i\cdot\nu=0$ on $\pa\dom$.
Note that $\langle u^0,\phi\rangle = \sum_{i=1}^n\langle\Pi_N(u^0_i),\phi_i\rangle$ 
since $\phi\in H_N^n$.

We claim that $S:X_T\to X_T$ is a self-mapping and a contraction. 
The proof is similar to that one for \cite[Prop.~4]{DHJKN20}. 
The main difference is that the
definition of the stopping time is based here on the $L^2(\dom)$ norm, while
the $H^1(\dom)$ norm was used in \cite{DHJKN20}. To compensate the weaker norm,
we exploit the Laplace structure of \eqref{1.eq}. Indeed, for the self-mapping
property, we need to verify that $\|S(u)\|_{X_{T\wedge\tau_R}}\le
C(\|u\|_{X_{T\wedge\tau_R}})$. Since only the (normally) elliptic term is different,
it is sufficient to estimate
\begin{align*}
  \E\bigg(&\sup_{0<t<T\wedge\tau_R}\bigg|\int_0^{t}\bigg\langle 
	\bigg(a_{i0}u_i(s)+\sum_{j=1}^n a_{ij}u_i(s)u_j(s)\bigg),\Delta\phi_i
	\bigg\rangle\d s\bigg|\bigg)^2 \\
  &\le T^2(1+C(R))\|\Delta\phi_i\|_{L^\infty(\dom)}^2 
	\E\|u(t)\|_{L^\infty(0,T;L^2(\dom))}^2 \\
	&\le C(T,R)\|\phi\|_{H_N}^2\E\|u(t)\|_{L^\infty(0,T;L^2(\dom))}^2,
\end{align*}
and the other terms are essentially
treated as in the proof of \cite[Prop.~4]{DHJKN20}.
For the contraction property, we need to estimate the difference
$S(u)-S(v)$ for $u$, $v\in X_{T\wedge\tau_R}$. It is sufficient to consider the term
\begin{align*} 
  \E\bigg(&\sup_{0<t<T\wedge\tau_R}\bigg|\int_0^t\bigg\langle
	a_{i0}(u_i-v_i)(s)+\sum_{j=1}^n a_{ij}\big((u_i-v_i)u_j+v_i(u_j-v_j)\big)(s)\bigg),
	\Delta\phi_i\bigg\rangle\d s\bigg|\bigg)^2 \\
  &\le C(R)T\|\Delta\phi_i\|_{L^\infty(\dom)}^2\E\int_0^{T\wedge\tau_R}
	\|u(s)-v(s)\|_{L^2(\dom)}^2\d s \\
	&\le C(N,R)T^2\|\phi\|_{H_N}^2\|u-v\|_{X_{T\wedge\tau_R}}^2.
\end{align*}
The remaining terms are estimated as in the proof of \cite[Prop.~4]{DHJKN20}.
This leads to
$$
  \|S(u)-S(v)\|_{X_{T\wedge\tau_R}} \le C(N,R)T\|u-v\|_{X_{T\wedge\tau_R}},
$$
showing that $S:X_{T^*}\to X_{T^*}$ is a contraction for $0<T^*<T\wedge\tau_R$
satisfying $C(N,R)T^*<1$. This shows that \eqref{2.sde}--\eqref{2.ic} possesses
a unique solution $u^{(N)}$ up to the stopping time $\tau_R$.

\subsection{Wong--Zakai approximation}\label{sec.wz}

We prove the existence of global-in-time solutions to another approximate
system of \eqref{1.eq} by replacing the Wiener process by the Wong--Zakai
approximation, leading to a system of ordinary differential equations.
This step is necessary to obtain the nonnegativity of the solutions $u_i^{(N)}$
constructed in the previous subsection.

We project \eqref{1.eq} as in the previous subsection onto the Galerkin space $H_N$
and introduce a uniform partition of the time interval $[0,T]$ with time step
$\eta=T/M$, where $M\in\N$. We set $t_k=k\eta$ for $k=0,\ldots,M$. The Wiener
process is approximated by the process \cite{WoZa65}
$$
  W^{(\eta)}_j(t) = W_j(t_k) + \frac{t-t_k}{\eta}\big(W_j(t_{k+1})-W_j(t_k)\big),
	\quad t\in[t_k,t_{k+1}],\ k=0,\ldots,M.
$$
The approximate system is given by
\begin{align}
  & \frac{\d u^{(N,\eta)}}{\d t} = \Pi_N\diver\bigg(\sum_{j=1}^n A_{ij}(u^{(N,\eta)})
	\na u_j^{(N,\eta)}\bigg) + f_i(u^{(N,\eta)},t), \quad\mbox{where} \label{2.ode} \\
	& f(u^{(N,\eta)},t) = \sum_{j=1}^n\Pi_N\big(\sigma_{ij}(u^{(N,\eta)})\big)
	\frac{\d W_j^{(\eta)}}{\d t}(t) - \frac12\Pi_N\sum_{j,k=1}^n
	\frac{\pa\sigma_{ij}}{\pa u_k}(u^{(N,\eta)})\sigma_{kj}(u^{(N,\eta)}) \nonumber
\end{align}
with the initial condition $u^{(N,\eta)}(0)=\Pi_N(u_i^0)$.
System \eqref{2.ode} can be written in the weak form
\begin{align}
  \langle u_i^{(N,\eta)}(t),\phi_i\rangle 
	&= \langle u_i^0,\phi_i\rangle - \int_0^t\sum_{j=1}^n\big\langle 
	A_{ij}(u^{(N,\eta)}(s))
	\na u_j^{(N,\eta)}(s),\na\phi_i\big\rangle \d s \nonumber \\
	&\phantom{xx}{}+ \int_0^t\sum_{j=1}^n\bigg\langle\sigma_{ij}(u^{(N,\eta)}(s))
	\frac{\d W_j^{(\eta)}}{\d t}(s),\phi_i\bigg\rangle \d s \nonumber \\
	&\phantom{xx}{}- \frac12 \int_0^t \bigg\langle \sum_{j,k=1}^n
	\frac{\pa\sigma_{ij}}{\pa u_k}(u^{(N,\eta)}(s))\sigma_{kj}(u^{(N,\eta)}(s)), 
	\phi_i\bigg\rangle \d s \label{WZ.weak}
\end{align}
for any $\phi_i\in H_N$. 
The last term is needed, since the Wong--Zakai approximation
converges to the Stratonovich noise that is related to the It\^o noise by
$$
  \sum_{j=1}^n\sigma_{ij}(u)\circ\d W_j(t) 
	= \sum_{j=1}^n\sigma_{ij}(u)\d W_j(t) + \frac12\sum_{j,k=1}^n
	\frac{\pa\sigma_{ij}}{\pa u_k}(u)\sigma_{kj}(u)\d t.
$$

We need to distinguish the cases with and without self-diffusion.
First, if $a_{ii}>0$ for $i=1,\ldots,n$, 
it follows from the techniques of \cite{CDJ18} (see \cite[Prop.~5]{DHJKN20} for
details) that for a.e.\ $\omega\in\Omega$,
there exists a global-in-time weak solution $u^{(N,\eta)}$ to \eqref{WZ.weak}
satisfying $u^{(N,\eta)}_i(\omega,\cdot,\cdot)\ge 0$ a.e.\ in $\Omega\times(0,T)$,
\begin{align*}
  & u_i^{(N,\eta)}(\omega,\cdot,\cdot)\in L^2(0,T;H^1(\dom)) 
	\cap L^\infty(0,T; L^1(\dom)) \cap L^{2+2/d}(Q_T), \\
  & \pa_t u_i^{(N,\eta)}(\omega,\cdot,\cdot)\in L^{\rho_2}(0,T;W^{1,\rho_2}(\dom)'), 
	\quad i=1,\ldots,n,
\end{align*}
where $\rho_2=(2d+2)/(2d+1)$, $u^{(N,\eta)}(0)=u^0$ in the sense
of $W^{1,2d+2}(\dom)'$, and \eqref{WZ.weak} is satisfied.

Second, if $a_{i0}>0$, $a_{ii}=0$ for $i=1,\ldots,n$, 
we conclude from the techniques of
\cite{ChJu06} that for a.e.\ $\omega\in\Omega$, there exists a global-in-time
weak solution $u^{(N,\eta)}$ to \eqref{WZ.weak}
satisfying $u^{(N,\eta)}_i(\omega,\cdot,\cdot)\ge 0$ a.e.\ in $\Omega\times(0,T)$
and the weak formulation \eqref{WZ.weak}. However, we obtain less regularity:
\begin{align*}
  & u_i^{(N,\eta)}(\omega,\cdot,\cdot)\in L^{\rho_1}(0,T;W^{1,\rho_1}(\dom)) 
	\cap L^\infty(0,T; L^1(\dom)) \cap L^{1+2/d}(Q_T), \\
  & \pa_t u_i^{(N,\eta)}(\omega,\cdot,\cdot)\in L^{\rho_2}(0,T;W^{1,\rho_2}(\dom)'), \\
	& (u_i^{(N,\eta)}u_j^{(N,\eta)})(\omega,\cdot,\cdot)\in
	L^{\rho_2}(0,T;W^{1,\rho_2}(\dom)),
\end{align*}
for $i=1,\ldots,n$, $j\neq i$, where $\rho_1=(d+2)/(d+1)$; 
see \cite{ChJu06} and the Appendix. 
In the weak formulation \eqref{WZ.weak}, we interpret the expression
$\sum_{j=1}^n A_{ij}(u^{(N,\eta)})\na u^{(N,\eta)}$ here as
$$
  a_{i0}\na u_i^{(N,\eta)} + \sum_{j=1,\,j\neq i}^n a_{ij}
	\na(u_i^{(N,\eta)}u_j^{(N,\eta)}).
$$
The nonnegativity of $u_i^{(N,\eta)}$ is a consequence of the entropy
method (see, e.g., \cite{CDJ18}) applied to the weak formulation \eqref{WZ.weak}. 
This formulation is important 
since the initial datum associated to the strong formulation 
\eqref{2.ode} is projected to the Galerkin space and $\Pi_N(u_i^0)$ may have no sign.
In the weak formulation, the projection is taken care of the test function
and we are allowed to work with the nonnegative initial datum $u_i^0$.

The proof in \cite{ChJu06,CDJ18}
provides a priori estimates for $u^{(N,\eta)}$ via the entropy inequality, but
they depend on $\eta$ because of the dependence of the source term $f_j$ on $\eta$.
Still, it is possible to pass to the limit $\eta\to 0$, since the solution to an
ordinary differential equation involving the Wong--Zakai approximation converges 
in mean to the solution to the corresponding stochastic differential equation
\cite[Chapter 6, Theorem 7.1]{IkWa89}. We can apply this result since
the nonlinearities in the strong form associated to \eqref{WZ.weak} are
Lipschitz continuous (not uniform in $N$).
We conclude that $u^{(N,\eta)}\to u^{(N)}$ in probability up to the stopping
time $\tau_R$ as $\eta\to 0$, where $u^{(N)}$ is the unique solution to
\eqref{2.sde}--\eqref{2.ic}. We deduce that $u^{(N)}_i(x,t)\ge 0$
for a.e.\ $(x,t)\in\dom\times(0,T\wedge\tau_R)$ $\Prob$-a.s.\ and $i=1,\ldots,n$. 

\begin{remark}\rm
The Wong--Zakai approximation is only needed to conclude the nonnegativity of
$u_i^{(N)}$. Another approach is to apply a stochastic version of the
Stampacchia truncation method; see \cite{CPT16}. Generally, maximum principle
arguments do not apply to cross-diffusion systems. For the present system, however,
this is possible since the off-diagonal diffusion coefficients in \eqref{1.A}
and the stochastic diffusion term vanish when $u_i=0$. We leave the details 
to the reader.
\qed
\end{remark}

\subsection{Entropy estimates}

We prove some estimates uniform in the Galerkin dimension $N$ showing that
the solution is actually global in time. The starting point
is a stochastic version of the entropy inequality.

\begin{lemma}[Entropy inequality]\label{lem.ei}
The solution $u^{(N)}$ to \eqref{2.sde}--\eqref{2.ic} satisfies for 
$0<t<T\wedge\tau_R$,
\begin{align}
  \E & \int_\dom h(u^{(N)}(t))\d x 
	+ 2\E\int_0^t\int_\dom\sum_{i=1}^n \pi_i\big(2a_{i0}|\na(u_i^{(N)})^{1/2}|^2
	+ a_{ii}|\na u_i^{(N)}|^2\big)\d x\d s \nonumber \\
  &{}+ 2\E\int_0^t\int_\dom\sum_{i,j=1,\,j\neq i}^n\pi_ia_{ij}
	|\na(u_i^{(N)}u_j^{(N)})^{1/2}|^2\d x\d s
	\le C(T) + C(T)\E\int_\dom h(u^{(N)}(0)^+)\d x, \label{2.ei}
\end{align}
where $C(T)>0$ depends on $T$ but not on $N$ or $R$ and 
$(u^{(N)}(0)^+)_i=\max\{0,u_i^{(N)}(0)\}$.
\end{lemma}

\begin{proof}
Let $u^{(N)}$ be the solution to \eqref{2.sde}--\eqref{2.ic} up to the stopping time
$\tau_R$. 
Since the entropy density defined in \eqref{1.h} is not a $C^2$ function
on $[0,\infty)^n$, we cannot apply the It\^o lemma to this function, and we need
to regularize. Let $\delta>0$ and define 
\begin{align*}
  h_\delta(u) &= \sum_{i=1}^n \pi_i\big((u_i+\delta)(\log(u_i+\delta)-1)+1\big)\quad
	\mbox{for }u\in[0,\infty)^n, \\
	h_{\delta}^+(u) &= \sum_{i=1}^n \pi_i
	\big((g_\delta(u_i) +\delta)(\log(g_\delta(u_i) +\delta)-1)+1\big)
	\quad\mbox{for }u\in\R^n,
\end{align*}
where $g_\delta$ is a smooth regularization of $z^+=\max\{0,z\}$ such that 
$g_\delta(z)\to z^+$ as $\delta\to 0$,
$g_\delta(z)+\delta>0$ for $z\in\R$, and $g_\delta(z)=z$ for $z\ge 0$. 
Then $h_\delta\in C^2([0,\infty)^n;$ $[0,\infty))$ and
$h_\delta^+\in C^2(\R^n;[0,\infty))$.
Note that these regularizations are different from that one used in \cite{DHJKN20}. 
Since $u_i^{(N)}(t)\ge 0$ for $t>0$ $\Prob$-a.s., we have by definition
$h_\delta^+(u^{(N)}(t))=h_\delta(u^{(N)}(t))$ $\Prob$-a.s.
The second regularization $h_\delta^+$ 
is needed since $u_i^{(N)}(0)=\Pi_N(u_i^0)$ may have no sign.
Because of It\^o's lemma and $g_\delta'(u_i^{(N)}(t))=1$, 
$g_\delta''(u_i^{(N)}(t))=0$ for $t>0$, we find that for $t>0$,
\begin{align}
  \int_\dom & h_\delta(u^{(N)}(t\wedge\tau_R))\d x 
	- \int_\dom h_\delta^+(u^N(0))\d x \nonumber \\
	&= -\int_0^{t\wedge\tau_R}\int_\dom\sum_{i=1}^n\pi_i a_{0i}
	\frac{|\na u_i^{(N)}|^2}{u_i^{(N)}+\delta}\d x\d s 
	- \int_0^{t\wedge\tau_R}\int_\dom\sum_{i,j=1}^n\pi_ia_{ij}u_j^{(N)}
	\frac{|\na u_i^{(N)}|^2}{u_i^{(N)}+\delta}\d x\d s \nonumber \\
	&\phantom{xx}{}- \int_0^{t\wedge\tau_R}\int_\dom\sum_{i,j=1}^n\pi_i 
	a_{ij}\frac{u_i^{(N)}}{u_i^{(N)}+\delta}\na u_i^{(N)}\cdot\na u_j^{(N)}\d x\d s 
	\nonumber \\
  &\phantom{xx}{}
	+ \int_0^{t\wedge\tau_R}\sum_{i,j=1}^n\bigg(\int_\dom\pi_i\sigma_{ij}(u^{(N)})
	\log(u_i^{(N)}+\delta)\d x\bigg)\d W_j(t) \nonumber \\
	&\phantom{xx}{}+ \frac12\int_0^{t\wedge\tau_R}\int_\dom\sum_{i,j=1}^n\pi_i
	\frac{\sigma_{ij}(u^{(N)})^2}{u_i^{(N)} + \delta}\d x\d s. \label{2.aux0}
\end{align}
We take the expectation on both sides and observe that the expectation of the
It\^o integral vanishes:
\begin{align}
  \E & \int_\dom h_\delta(u^{(N)}(t\wedge\tau_R))\d x 
	- \E\int_\dom h_\delta^+(u^N(0))\d x \nonumber \\
	&= -\E\int_0^{t\wedge\tau_R}\int_\dom\sum_{i=1}^n\pi_i a_{0i}
	\frac{|\na u_i^{(N)}|^2}{u_i^{(N)}+\delta}\d x\d s 
	- \E\int_0^{t\wedge\tau_R}\int_\dom\sum_{i,j=1}^n\pi_ia_{ij}u_j^{(N)}
	\frac{|\na u_i^{(N)}|^2}{u_i^{(N)}+\delta}\d x\d s \nonumber \\
	&\phantom{xx}{}- \E\int_0^{t\wedge\tau_R}\int_\dom\sum_{i,j=1}^n\pi_i 
	a_{ij}\frac{u_i^{(N)}}{u_i^{(N)}+\delta}\na u_i^{(N)}\cdot\na u_j^{(N)}\d x\d s 
	\nonumber \\
	&\phantom{xx}{}+ \frac12\E\int_0^{t\wedge\tau_R}\int_\dom\sum_{i,j=1}^n\pi_i
	\frac{\sigma_{ij}(u^{(N)})^2}{u_i^{(N)} + \delta}\d x\d s 
	=: I_1^\delta + \cdots + I_4^\delta. \label{2.aux}
\end{align}

We wish to perform the limit $\delta\to 0$ in \eqref{2.aux}. By continuity,
$$
  h_\delta(u^{(N)}(\omega,x,t\wedge\tau_R))\to h(u^{(N)}(\omega,x,t\wedge\tau_R))
	\quad\mbox{for a.e. }(\omega,x,t)\in\Omega\times\dom\times(0,T).
$$
Moreover, for given $(\omega,x)\in\Omega\times\dom$, there exists $C>0$ such that
for all $\delta>0$,
$$
  \big(u_i^{(N)}(\log u_i^{(N)}-1)+1\big)(\omega,x,t\wedge\tau_R)
	\le C\big(1+u^{(N)}(\omega,x,t\wedge\tau_R)^2\big),
$$
and the right-hand side is uniformly integrable in $\Omega\times\dom$ for a fixed
$t\in[0,T\wedge\tau_R]$ (because of the definition of the stopping time). 
We conclude from the dominated convergence theorem that
\begin{align*}
  \E\int_\dom h_\delta(u^{(N)}(t\wedge\tau_R))\d x
	&\to \E\int_\dom h(u^{(N)}(t\wedge\tau_R))\d x, \\
  \E\int_\dom  h_\delta^+(u^N(0))\d x &\to \E\int_\dom h(u^{(N)}(0)^+)\d x
	\quad\mbox{as }\delta\to 0,
\end{align*}
recalling that $(u^{(N)}(0)^+)_i = \max\{0,u_i^{(N)}(0)\}$.
The limit $\delta\to 0$ in $I_1^\delta$, $I_2^\delta$, and $I_4^\delta$ can be
performed because of the monotone convergence theorem, while the dominated
convergence theorem allows us to pass to the limit in $I_3^\delta$. Then
the limit $\delta\to 0$ in \eqref{2.aux} leads to
\begin{align}
  \E & \int_\dom h(u^{(N)}(t\wedge\tau_R))\d x 
	- \E\int_\dom h(u^{(N)}(0)^+)\d x \nonumber \\
  &= -\E\int_0^{t\wedge\tau_R}\int_\dom\sum_{i=1}^n\pi_i a_{0i}
	\frac{|\na u_i^{(N)}|^2}{u_i^{(N)}}\d x\d s
	- \E\int_0^{t\wedge\tau_R}\int_\dom\sum_{i,j=1}^n\pi_i a_{ij}
	u_j^{(N)}\frac{|\na u_i^{(N)}|^2}{u_i^{(N)}}\d x\d s \nonumber \\
	&\phantom{xx}{}- \E\int_0^{t\wedge\tau_R}\int_\dom\sum_{i,j=1}^n\pi_i 
	a_{ij}\na u_i^{(N)}\cdot\na u_j^{(N)}\d x\d s 
	+ \frac12\E\int_0^{t\wedge\tau_R}\int_\dom\sum_{i,j=1}^n\pi_i
	\frac{\sigma_{ij}(u^{(N)})^2}{u_i^{(N)}}\d x\d s. \label{2.aux2}
\end{align}
Because of the detailed-balance condition \eqref{2.dbc}, 
$(\pi_ia_{ij})$ is symmetric. Thus, the second and third integrand on the right-hand
side can be formulated as
\begin{align*}
  \sum_{i,j=1}^n & \pi_ia_{ij}\bigg(u_j^{(N)}\frac{|\na u_i^{(N)}|^2}{u_i^{(N)}}
	+ \na u_i^{(N)}\cdot\na u_j^{(N)}\bigg) \\
	&= \frac12\sum_{i,j=1}^n \pi_ia_{ij}\bigg(u_j^{(N)}\frac{|\na u_i^{(N)}|^2}{u_i^{(N)}}
	+ u_i^{(N)}\frac{|\na u_j^{(N)}|^2}{u_j^{(N)}}
	+ 2\na u_i^{(N)}\cdot\na u_j^{(N)}\bigg) \\
	&= \frac12\sum_{i,j=1}^n\pi_ia_{ij}u_i^{(N)}u_j^{(N)}\big(
	|\na\log u_i^{(N)}|^2 + |\na\log u_j^{(N)}|^2 
	+ 2\na\log u_i^{(N)}\cdot\na\log u_j^{(N)}\big) \\
	&= \frac12\sum_{i,j=1}^n\pi_ia_{ij}u_i^{(N)}u_j^{(N)}
	\big|\na\log(u_i^{(N)}u_j^{(N)})\big|^2
	= 2\sum_{i,j=1}^n\pi_ia_{ij}\big|\na(u_i^{(N)}u_j^{(N)})^{1/2}\big|^2.
\end{align*}
By Assumption (A5), the last integral in \eqref{2.aux2} is estimated according to
$$
  \frac12\E\int_0^{t\wedge\tau_R}\int_\dom\sum_{i,j=1}^n\pi_i
	\frac{\sigma_{ij}(u^{(N)})^2}{u_i^{(N)}}\d x\d s
	\le C\int_0^{t\wedge\tau_R}\int_\dom(1 + h(u^{(N)}))\d x\d s.
$$
Inserting these expressions into \eqref{2.aux2} and applying Gronwall's inequality
gives
\begin{align*}
  \E & \int_\dom h(u^{(N)}(t\wedge\tau_R))\d x + 4\E\int_0^{t\wedge\tau_R}\int_\dom
	\sum_{i=1}^n\pi_ia_{i0}|\na(u_i^{(N)})^{1/2}|^2\d x\d s \\
	&{}+ 2\E\int_0^{t\wedge\tau_R}\int_\dom\sum_{i,j=1}^n\pi_ia_{ij}
	|\na(u_i^{(N)}u_j^{(N)})^{1/2}|^2\d x\d s
	\le C(T) + C(T)\E\int_\dom h(u^{(N)}(0)^+)\d x,
\end{align*}
where $C(T)>0$ is independent of $N$ and $R$. Consequently, the right-hand side
does not depend on the chosen sequence of stopping times $\tau_R$, and we can pass
to the limit $R\to \infty$.
The limit $u_i^{(N)}(0)=\Pi_N(u_i^0)\to u_i^0\ge 0$ in $L^2(\dom)$ as $N\to\infty$
yields $h(u^{(N)}(0)^+)\to h(u^0)$ in $L^1(\dom)$. Thus, the right-hand side of
\eqref{2.ei} is independent of $N$ and $R$.
\end{proof}

The entropy inequality in Lemma \ref{lem.ei} provides a uniform bound for
$\sup_{0\le t\le T}\E\|u^{(N)}(t)\|_{L^1(\dom)}$ but we need
a uniform bound for $\E(\sup_{0<t<T}\|u^{(N)}(t)\|_{L^1(\dom)})$,
which will be used later to obtain higher order moment estimates.
This is shown in the following lemma.

\begin{lemma}
The solution $u^{(N)}$ to \eqref{2.sde}--\eqref{2.ic} satisfies the following bounds:
\begin{align}
  \sup_{N\in\N}\E\|u^{(N)}\|_{L^\infty(0,T;L^1(\dom))} &\le C(u^0,T), \label{2.L1} \\
	\sup_{N\in\N}\sum_{i=1}^n\big(a_{i0}\E\|(u_i^{(N)})^{1/2}\|_{L^2(0,T;H^1(\dom))}^2
	+ a_{ii}\E\|u_i^{(N)}\|_{L^2(0,T;H^1(\dom))}^2\big) &\le C(u^0,T), \label{2.H1} \\
	\sup_{N\in\N}\sum_{j\neq i}a_{ij}\E\|\na(u_i^{(N)}u_j^{(N)})^{1/2}
	\|_{L^2(0,T;L^2(\dom))}^2 &\le C(u^0,T). \label{2.sqrt2}
\end{align}
In particular, the solution $u^{(N)}$ is global in time for $d\ge 1$
(with self-diffusion) and for $d\le 3$ (without self-diffusion). 
\end{lemma}

\begin{proof}
Let $u^{(N)}$ be the solution to \eqref{2.sde}--\eqref{2.ic} up to the stopping time
$\tau_R$ and let $T<\tau_R$. The starting point of the proof is equation \eqref{2.aux0}.
Instead of taking first the expectation as in the proof of
Lemma \ref{lem.ei}, we pass first to the limit $\delta\to 0$. This can be
done as in Lemma \ref{lem.ei} except for the stochastic integral. We claim that
\begin{align}
  \int_0^{T} & \int_\dom\sum_{i,j=1}^n\sigma_{ij}(u^{(N)}(s))
	\log(u_i^{(N)}(s)+\delta)\d x\d W_j(s) \nonumber \\
  &\to \int_0^{T} \int_\dom\sum_{i,j=1}^n\sigma_{ij}(u^{(N)}(s))
	\log(u_i^{(N)}(s))\d x\d W_j(s) \label{2.stoch}
\end{align}
as $\delta\to 0$. 
To prove this limit, we use the stochastic dominated convergence theorem
\cite[Theorem 6.44]{Ebe19}. For this, let
\begin{align*}
  F_\delta(t) &= \int_\dom f_\delta(x)\d x
	= \int_\dom \sum_{i=1}^n\sigma_{ij}(u^{(N)}(t))\log(u_i^{(N)}(t)+\delta)\d x, \\
  F(t) &= \int_\dom f(x)\d x = \int_\dom\sum_{i=1}^n\sigma_{ij}(u^{(N)}(t))
	\log u_i^{(N)}(t)\d x.
\end{align*}
It is clear that $f_\delta(x)\to f(x)$ a.e.\ in $\dom$.
We wish to find an integrable function $g$ such that $|f_\delta(x)|\le g(x)$ for
$x\in\dom$. Let $\delta\in(0,1)$. If $z\in[0,1-\delta)$, we have
$|\log(z+\delta)|\le|\log z|$. If $z\in[1-\delta,1)$, it follows that
$|\log(z+\delta)|\le \log 2$. Finally, if $z>1$, 
$$
  \log(z+\delta) = \int_1^z\frac{\d r}{r} + \int_z^{z+\delta}\frac{\d r}{r}
	\le \int_1^z\frac{\d r}{r} + \int_1^{1+\delta}\frac{\d r}{r}
	= \log z + \log(1+\delta).
$$
Therefore, in view of Assumption (A5) and the entropy inequality in
Lemma \ref{lem.ei}, the function 
$$
  g(x)=\sum_{i=1}^n\sigma_{ij}(u^{(N)}(x,t))(\log u_i^{(N)}+\log 2)
$$
is integrable in $\dom$. We deduce from the Lebesgue dominated convergence theorem
that $F_\delta(t)\to F(t)$ as $\delta\to 0$.
By the definition of $g$ and Assumption (A5), 
we can dominate $F_\delta$ pointwise for any $\delta>0$ according to
\begin{equation}\label{2.FG}
  |F_\delta(t)|\le G(x) := C_h\bigg(\int_\dom(1+h(u^{(N)}(x,t)))\d x\bigg)^{1/2}
	+ C\int_\dom u_i^{(N)}(x,t)\d x + C,
\end{equation}
and $G$ is square-integrable, since
$$
  \|G\|_{L^2(Q_T)}^2
	\le C(T) + C\E\int_0^T\int_\dom\big(h(u^{(N)}(x,t)) + |u^{(N)}(x,t)|^2\big)\d x\d s
	< \infty.
$$
By the stochastic dominated convergence theorem, we infer from the pointwise 
convergence $F_\delta(t)\to F(t)$ and the bound \eqref{2.FG} that \eqref{2.stoch} holds,
proving the claim.

Repeating the calculations following \eqref{2.aux2}, we obtain
\begin{align}
  \int_\dom & h(u^{(N)}(t))\d x + 4\int_0^{t}\int_\dom
	\sum_{i=1}^n\pi_i a_{i0}|\na(u_i^{(N)})^{1/2}|^2\d x\d s \nonumber \\
	&\phantom{xx}{}+ 2\int_0^{t}\int_\dom\sum_{i,j=1}^n\pi_i a_{ij}
	|\na(u_i^{(N)}u_j^{(N)})^{1/2}|^2\d x\d s \nonumber \\
	&\le \int_\dom h(u^0)\d x + \int_0^{t}\int_\dom\sum_{i,j=1}^n
	\sigma_{ij}(u^{(N)})\log u_i^{(N)}\d x\d W_j(s) \nonumber \\
	&\phantom{xx}{}+ \frac12\int_0^{t}\int_\dom\sum_{i,j=1}^n\pi_i
	\frac{\sigma_{ij}(u^{(N)})^2}{u_i^{(N)}}\d x\d s. \label{2.aux3}
\end{align}
We take the supremum over $0<t<T$ and the expectation and 
apply the Burkholder--Davis--Gundy inequality:
\begin{align*}
  \E & \sup_{0<t<T}\int_\dom h(u^{(N)}(t))\d x
	+ 4\E\int_0^{T}\int_\dom
	\sum_{i=1}^n\pi_i a_{i0}|\na(u_i^{(N)})^{1/2}|^2\d x\d s \\
  &\phantom{xx}{}+ 2\E\int_0^{T}\int_\dom\sum_{i,j=1}^n\pi_i a_{ij}
	|\na(u_i^{(N)}u_j^{(N)})^{1/2}|^2\d x\d s \\
	&\le \E\int_\dom h(u^0)\d x 
	+ \E\bigg(\int_0^{T}\int_\dom\sum_{i,j=1}^n\big(\sigma_{ij}(u^{(N)})
	\log u_i^{(N)}\big)^2\d x\d s\bigg)^{1/2} \\
	&\phantom{xx}{}+ \frac12\E\int_0^{T}\int_\dom\sum_{i,j=1}^n\pi_i
	\frac{\sigma_{ij}(u^{(N)})^2}{u_i^{(N)}}\d x\d s \\
	&\le \E\int_\dom h(u^0)\d x + C\E\int_0^{T}
	\bigg(1+ \int_\dom h(u^{(N)})\d x\bigg)\d s,
\end{align*}
where we used Assumption (A5) in the last step. 
By Fubini's theorem and Gronwall's lemma, we conclude that
$$
  \E\sup_{0<t<T}\int_\dom h(u^{(N)}(t))\d x
	\le C(T) + C(T)\E\int_\dom h(u^{(N)}(0)^+)\d x\quad\mbox{for }0<T<\tau_R,
$$
where $C(T)>0$ is independent of $N$ and $R$. Passing to the limit $R\to\infty$
results in
$$
  \E\sup_{0<t<T}\int_\dom h(u^{(N)}(t))\d x\le C(T) 
	+ C(T)\E\int_\dom h(u^{(N)}(0)^+)\d x \le C(u^0,T).
$$
for $T>0$. Since the entropy density dominates the $L^1$ norm, this shows that
$$
  \sup_{N\in\N}\E\Big(\sup_{0<t<T}\|u^{(N)}(t)\|_{L^1(\dom)}\Big)\le C(u^0,T).
$$
Estimate \eqref{2.H1} is obtained from the Poincar\'e--Wirtinger inequality, 
the previous estimate, and the gradient estimate in \eqref{2.ei}.
Moreover, \eqref{2.sqrt2} also follows from \eqref{2.ei}.

It remains to show that $u^{(N)}$ is global in time. 
In case with self-diffusion, estimate \eqref{2.H1} immediately implies that
$\E\|u_i^{(N)}\|_{L^2(0,T;L^2(\dom))}^2\le C$.
In case without self-diffusion, we deduce from the Gagliardo--Nirenberg inequality
with $\theta=d/4$, the H\"older inequality with $p=4/d$, $q=4/(4-d)$
(such that $1/p+1/q=1$), and
estimates \eqref{2.L1} and \eqref{2.H1} that
\begin{align}
  \E&\|u_i^{(N)}\|_{L^{4/d}(0,T;L^2(\dom))}
  = \E\bigg(\int_0^T\|(u_i^{(N)})^{1/2}\|_{L^4(\dom)}^{8/d}\d t\bigg)^{d/4} 
	\nonumber \\
	&\le  C\E\bigg(\int_0^T\|(u_i^{(N)})^{1/2}\|_{H^1(\dom)}^{8\theta/d}
	\|(u_i^{(N)})^{1/2}\|_{L^2(\dom)}^{8(1-\theta)/d}\d t\bigg)^{d/4} \nonumber \\
	&\le C\E\bigg\{\|u_i^{(N)}\|_{L^\infty(0,T;L^1(\dom))}^{1-d/4}
	\bigg(\int_0^T\|(u_i^{(N)})^{1/2}\|_{H^1(\dom)}^2 \d t\bigg)^{d/4}\bigg\} 
	\nonumber \\
	&\le \Big(\E\|u_i^{(N)}\|_{L^\infty(0,T;L^1(\dom))}\Big)^{1-d/4}
	\bigg\{\E\int_0^T\|(u_i^{(N)})^{1/2}\|_{H^1(\dom)}^2\d t
	\bigg\}^{d/4} \le C. \label{2.4d.no}
\end{align}
At this point, we need the restriction $d\le 3$.
As the $L^2(\dom)$ is controlled in both cases, the stopping time $\tau_R$ 
equals the final time $T$, and the solution $u^{(N)}$ is global in time.
\end{proof}

\subsection{Further uniform estimates}

Next, we show some estimates for higher-order moments. This step was not necessary in
\cite{DHJKN20}, since the solutions in that paper are bounded.

\begin{lemma}[Higher-order moments]\label{lem.mom}
Let $u^{(N)}$ be the solution to \eqref{2.sde}--\eqref{2.ic} and let $p\ge 2$. Then,
for any $i=1,\ldots,n$,
\begin{align}
  \sup_{N\in\N}\E\|u^{(N)}\|_{L^\infty(0,T;L^1(\dom))}^p
	&\le C(p,u^0,T), \label{2.hp} \\
	\sup_{N\in\N}\big(a_{i0}\E
	\|(u_i^{(N)})^{1/2}\|_{L^2(0,T;H^1(\dom))}^p + a_{ii}\E
	\|u_i^{(N)}\|_{L^2(0,T;H^1(\dom))}^p\big) &\le C(p,u^0,T), \label{2.nablap}
\end{align}
where $C(p,u^0,T)>0$ does not depend on $N$.
\end{lemma}

\begin{proof}
We raise \eqref{2.aux3} to the power $p\ge 2$, take the expectation, apply
the Burkholder--Davis--Gundy inequality to the stochastic term, and use Assumption 
(A5) to find that 
\begin{align*}
  \E & \bigg(\int_\dom h(u^{(N)}(t))\d x\bigg)^p 
	+ 4\E\bigg(\int_0^T\int_\dom\sum_{i=1}^n\pi_i a_{i0}
	|\na(u_i^{(N)})^{1/2}|^2\d x\d s\bigg)^p \\
	&\phantom{xx}{}+ 2\E\bigg(\int_0^T\int_\dom\sum_{i,j=1}^n\pi_i a_{ij}
	|\na(u_i^{(N)}u_j^{(N)})^{1/2}|^2\d x\d s\bigg)^p \\
	&\le C(p,u^0)
	+ C\E\bigg(\int_0^T\sum_{i,j=1}^n\|\sigma_{ij}(u^{(N)})\log u_i^{(N)}\|_{L^2(\dom)}^2
	\d s\bigg)^{p/2} \\
	&\phantom{xx}{}+ C\E\bigg(\int_0^T\int_\dom\sum_{i,j=1}^n\pi_i
	\frac{\sigma_{ij}(u^{(N)})^2}{u_i^{(N)}}\d x\d s\bigg)^p \\
	&\le C(p,u^0)
	+ C\E\bigg(\int_0^T\int_\dom(1+h(u^{(N)}))\d x\d s\bigg)^{p/2} \\
	&\phantom{xx}{}+ C\E\bigg(\int_0^T\int_\dom(1+h(u^{(N)}))\d x\d s\bigg)^p.
\end{align*}
The lemma follows after applying Jensen's and
Gronwall's inequality, using the fact that the entropy density dominates
the $L^1$ norm, and applying the Poincar\'e--Wirtinger inequality.
\end{proof}

We derive further higher-order moment estimates from Lemma \ref{lem.mom}.
For this, we distinguish the cases with and without self-diffusion.

\begin{lemma}[Higher-order moments, with self-diffusion]
Let $a_{ii}>0$ for all $i=1,\ldots,n$, let $u^{(N)}$ be the solution to 
\eqref{2.sde}--\eqref{2.ic}, and let $p\ge 2$. Then for any $i=1,\ldots,n$,
\begin{align}
	\sup_{N\in\N}\E\|u_i^{(N)}\|_{L^{2+2/d}(Q_T)}^p 
	&\le C(p,u^0,T), \label{2.Ld} \\
	\sup_{N\in\N}\E\|u_i^{(N)}(t)\|_{L^{2+4/d}(0,T;L^2(\dom))}^p 
	&\le C(p,u^0,T). \label{2.L3}
\end{align}
\end{lemma}

\begin{proof}
Applying the Gagliardo--Nirenberg
inequality as in \cite[p.~95]{Jue16} yields \eqref{2.Ld}. Estimate \eqref{2.L3}
is obtained from another application of the Gagliardo--Nirenberg inequality. 
For this, let $\theta=d/(d+2)$. Then $\theta(2+4/d)=2$ and
\begin{align*}
  \E&\bigg(\int_0^T\|u_i^{(N)}\|_{L^2(\dom)}^{2+4/d}\d t\bigg)^p
	\le C\E\bigg(\int_0^T\|u_i^{(N)}\|_{H^1(\dom)}^{\theta(2+4/d)}
	\|u_i^{(N)}\|_{L^1(\dom)}^{(1-\theta)(2+4/d)}\d t\bigg)^p \\
	&\le C(T)\E\bigg(\|u_i^{(N)}\|_{L^\infty(0,T;L^1(\dom))}^{4/d}
	\int_0^T\|u_i^{(N)}\|_{H^1(\dom)}^2\d t\bigg)^p \\
  &\le C(T)\Big(\E\|u_i^{(N)}\|_{L^\infty(0,T;L^1(\dom))}^{8p/d}\Big)^{1/2}
	\bigg\{\E\bigg(\int_0^T\|u_i^{(N)}\|_{H^1(\dom)}^2\d t\bigg)^{2p}\bigg\}^{1/2},
\end{align*}
and the conclusion follows from estimates \eqref{2.hp} and \eqref{2.nablap}. 
\end{proof}

\begin{lemma}[Higher-order moments, without self-diffusion]
Let $a_{i0}>0$ for all $i=1,\ldots,n$, let $u^{(N)}$ be the solution to 
\eqref{2.sde}--\eqref{2.ic}, and let $p\ge 2$. Then for any $i=1,\ldots,n$,
\begin{align}
	\sup_{N\in\N}\E\|u_i^{(N)}\|_{L^2(0,T;W^{1,1}(\dom))}^p 
	&\le C(p,u^0,T), \label{2.W11.no} \\
	\sup_{N\in\N}\E\|u_i^{(N)}\|_{L^{1+2/d}(Q_T)}^p
	&\le C(p,u^0,T), \label{2.L2.no} \\
	\sup_{N\in\N}\E\|u_i^{(N)}\|_{L^{4/d}(0,T;L^2(\dom))}^p 
	&\le C(p,u^0,T) \quad \text{for } d\leq 4,\label{2.4dL2.no} \\
	\sup_{N\in\N}\E\|u_i^{(N)}\|_{L^{\rho_1}(0,T;W^{1,\rho_1}(\dom))}^p
	&\le C(p,u^0,T), \label{2.W1r.no}
\end{align}
where $\rho_1=(d+2)/(d+1)$.
\end{lemma}

\begin{proof}
The identity $\na u_i^{(N)}=2(u_i^{(N)})^{1/2}\na(u_i^{(N)})^{1/2}$ and
the H\"older inequality show that
\begin{align*}
  \E\|\na u_i^{(N)}\|_{L^2(0,T;L^1(\dom))}^{p}
	&\le C\E\bigg(\int_0^T\|(u_i^{(N)})^{1/2}\|_{L^2(\dom)}^2
	\|\na (u_i^{(N)})^{1/2}\|_{L^2(\dom)}^2\d t\bigg)^{p/2} \\
  &\le C\E\bigg(\|u_i^{(N)}\|_{L^\infty(0,T;L^1(\dom))}\int_0^T
	\|\na (u_i^{(N)})^{1/2}\|_{L^2(\dom)}^2\d t\bigg)^{p/2} \\
	&\le C\big(\E\|u_i^{(N)}\|_{L^\infty(0,T;L^1(\dom))}^{p}\big)^{1/2}
	\big(\E\|\na (u_i^{(N)})^{1/2}\|_{L^2(0,T;L^2(\dom)}^{2p}\big)^{1/2}.
\end{align*}
Because of \eqref{2.hp} and \eqref{2.nablap}, the right-hand side is bounded.
Using \eqref{2.hp} again, we infer that \eqref{2.W11.no} holds.
Estimate \eqref{2.L2.no} is obtained
from the Gagliardo--Nirenberg inequality with $\theta=d/(d+2)$. Indeed,
taking into account estimates \eqref{2.hp} and \eqref{2.nablap},
\begin{align*}
  \E&\bigg(\int_0^T\|(u_i^{(N)})^{1/2}\|_{L^{2+4/d}(\dom)}^{2+4/d}\bigg)^p \\
	&\le C\E\bigg(\int_0^T\|(u_i^{(N)})^{1/2}\|_{H^1(\dom)}^{\theta(2+4/d)}
	\|(u_i^{(N)})^{1/2}\|_{L^2(\dom)}^{(1-\theta)(2+4/d)}\d t\bigg)^p \\
	&\le C\E\bigg(\|u_i^{(N)}\|_{L^\infty(0,T;L^1(\dom))}^{2/d}\int_0^T
	\|(u_i^{(N)})^{1/2}\|_{H^1(\dom)}^2\d t\bigg)^p \\
  &\le C\Big(\E\|u_i^{(N)}\|_{L^\infty(0,T;L^1(\dom))}^{4p/d}\Big)^{1/2}
	\bigg\{\E\bigg(\int_0^T\|(u_i^{(N)})^{1/2}\|_{H^1(\dom)}^{2}\bigg)^{2p}
	\bigg\}^{1/2} \le C.
\end{align*}
Estimate \eqref{2.4dL2.no} can be shown as in \eqref{2.4d.no}.
Finally, estimate \eqref{2.W1r.no} is a consequence of
$\na u_i^{(N)} = 2(u_i^{(N)})^{1/2}\na(u_i^{(N)})^{1/2}$, the H\"older inequality,
and estimates \eqref{2.nablap} and \eqref{2.L2.no}. 
\end{proof}

\begin{lemma}
Let $u^{(N)}$ be the solution to 
\eqref{2.sde}--\eqref{2.ic} and let $p\ge 2$. Then
for any $i=1,\ldots,n$ and $j\neq i$,
\begin{align}
  \sup_{N\in\N}\E\|(u_i^{(N)}u_j^{(N)})^{1/2}\|_{L^\infty(0,T;L^1(\dom))}^p
	&\le C(p,u^0,T), \label{2.dL1.no} \\
	\sup_{N\in\N}\E\|(u_i^{(N)}u_j^{(N)})^{1/2}\|_{L^2(0,T;H^1(\dom))}^p
	&\le C(p,u^0,T), \label{2.dH1.no} \\
	\sup_{N\in\N}\E\|(u_i^{(N)}u_j^{(N)})^{1/2}\|_{L^{2+2/d}(Q_T)}^p
	&\le C(p,u^0,T), \label{2.dL24.no} \\
	\sup_{N\in\N}\E\|u_i^{(N)}u_j^{(N)}\|_{L^{1+2/d}(0,T;L^1(\dom))}^p
	&\le C(p,u^0,T), \label{2.dL11.no} \\
	\sup_{N\in\N}\E\|\na(u_i^{(N)}u_j^{(N)})\|_{L^{\rho_1}(0,T;L^1(\dom))}^p
	&\le C(p,u^0,T), \label{2.dW21.no} \\
	\sup_{N\in\N}\E\|\na(u_i^{(N)}u_j^{(N)})\|_{L^{\rho_2}(Q_T)}^p
	&\le C(p,u^0,T), \label{2.dnabla.no}
\end{align}
where $\rho_1=(d+2)/(d+1)$ and $\rho_2=(2d+2)/(2d+1)$.
\end{lemma}

\begin{proof} 
The H\"older inequality and estimate \eqref{2.hp} yield immediately \eqref{2.dL1.no}.
By the Poincar\'e--Wirtinger inequality, estimates \eqref{2.sqrt2} and 
\eqref{2.dL1.no} lead to \eqref{2.dH1.no}. 
Estimates \eqref{2.dL24.no} and \eqref{2.dL11.no}
follow from the Gagliardo--Nirenberg inequality, taking into account 
estimates \eqref{2.dL1.no} and \eqref{2.dH1.no} (see \eqref{2.4d.no} for 
a similar proof). Finally, estimates \eqref{2.dL1.no} and \eqref{2.dH1.no}
imply that
$$
  \na(u_i^{(N)}u_j^{(N)}) = 2(u_i^{(N)}u_j^{(N)})^{1/2}\na(u_i^{(N)}u_j^{(N)})^{1/2}
$$
is bounded in $L^{\rho_1}(0,T;L^1(\dom))$ and in $L^{\rho_2}(Q_T)$,
verifying \eqref{2.dW21.no} and \eqref{2.dnabla.no}.
\end{proof}

\subsection{Fractional time regularity}

We show that the solution $u^{(N)}$ to \eqref{2.sde}--\eqref{2.ic} possesses
a uniform bound for a fractional time derivative. This result is used to
establish the tightness of the laws of $u^{(N)}$ in some Lebesgue spaces.
In our previous work \cite{DHJKN20}, 
the tightness of the laws of the approximate solutions was
proved in a different way by verifying the Aldous condition.
We recall the definition of the Sobolev--Slobodeckij spaces. Let $X$ be a vector
space and let $p\ge 1$, $\alpha\in(0,1)$. 
Then $W^{\alpha,p}(0,T;X)$ is the set of
all functions $v\in L^p(0,T;X)$ for which 
$$
  \|v\|_{W^{\alpha,p}(0,T;X)}^p = \|v\|_{L^p(0,T;X)}^p
	+ \int_0^T\int_0^T\frac{\|v(t)-v(s)\|_X^p}{|t-s|^{1+\alpha p}}\d t\d s
$$
is finite. With this norm, $W^{\alpha,p}(0,T;X)$ becomes a Banach space.
In the case without self-diffusion, we assume that $d\le 3$. 

\begin{lemma}[Time regularity]
Let $u^{(N)}$ be the solution to \eqref{2.sde}--\eqref{2.ic} and let 
$m\in\N$ satisfy $m>d/2+1$. 
Then there exists $C(u^0,T)>0$ such that
\begin{equation}\label{2.time}
  \sup_{N\in\N}\E\|u^{(N)}\|_{W^{\alpha,2}(0,T;H^{m}(\dom)')}^2 \le C(u^0,T),
\end{equation}
where $\alpha<1/2$ (with self-diffusion) and $\alpha<1/(d+2)$ (without
self-diffusion).
\end{lemma}

The time regularity of $u^{(N)}$ is restricted by the Sobolev regularity
of the stochastic integral; see, e.g., \cite[Lemma 2.1]{FlGa95}.

\begin{proof}
Estimate \eqref{2.L1} and the continuous embedding $L^1(\dom)\hookrightarrow
H^{m}(\dom)'$ show that the sequence $(\E\|u_i^{(N)}\|_{L^2(0,T;H^{m}(\dom)')}^2)$
is uniformly bounded. It remains to show that the following integral is finite:
\begin{align*}
  \E&\int_0^T\int_0^T\frac{\|u_i^{(N)}(t)-u_i^{(N)}(s)\|_{H^{m}(\dom)'}^2}{
	|t-s|^{1+2\alpha}}\d t\d s \\
  &\le \int_0^T\int_0^T|t-s|^{-1-2\alpha}\E\bigg\|\int_{s\wedge t}^{t\lor s}
	\diver\sum_{j=1}^n A_{ij}(u^{(N)}(r))\na u_j^{(N)}(r)\d r\bigg\|_{H^{m}(\dom)'}^2
	\d t\d s \\
	&\phantom{xx}{}+ \int_0^T\int_0^T|t-s|^{-1-2\alpha}\E\bigg\|\int_{s\wedge t}^{t\lor s}
	\sum_{j=1}^n\sigma_{ij}(u^{(N)}(r))\d W_j(r)\bigg\|_{H^{m}(\dom)'}^2\d t\d s \\
	&=: J_1 + J_2.
\end{align*}
Before we estimate $J_1$ and $J_2$, we recall a well-known result for
the sake of completeness.
Let $g\in L^1(0,T)$ and $\delta<2$, $\delta\neq 1$. We claim that
\begin{equation}\label{2.int}
  I := \int_0^T\int_0^T|t-s|^{-\delta}\int_{s\wedge t}^{t\lor s}g(r)\d r\d t\d s 
	< \infty.
\end{equation}
Indeed, a change of the integration domain and integration by parts gives
\begin{align}
  I &= 2\int_0^T\int_s^T(t-s)^{-\delta}\bigg(\int_s^t g(r)dr\bigg)\d t\d s \nonumber \\
  &= -\frac{2}{1-\delta}\int_0^T\int_s^T(t-s)^{1-\delta}g(t)\d t\d s
	+ \frac{2}{1-\delta}\int_0^T(T-s)^{1-\delta}\int_s^t g(r)\d r\d s, \label{2.I}
\end{align}
observing that $\lim_{t\to s}(t-s)^{1-\delta}\int_s^t g(r)\d r=0$ 
for $1-\delta>-1$, since the integrability of $g$ implies that 
$\lim_{t\to s}(t-s)^{-1}\int_s^t g(r)\d r=g(s)$ for a.e.\ $s$. The claim follows
as the integrals on the right-hand side of \eqref{2.I} are finite.

{\em Step 1: Case with self-diffusion.}
Let $a_{ii}>0$ for all $i=1,\ldots,n$. 
We need some preparations before we estimate $J_1$. We observe that
\begin{align*}
  \bigg\|\sum_{j=1}^n A_{ij}(u^{(N)})\na u_j^{(N)}\bigg\|_{L^1(\dom)}
	&= \bigg\|\bigg(a_{i0} + 2\sum_{j=1}^n a_{ij}u_j^{(N)}\bigg)\na u_i^{(N)}
	+ \sum_{j\neq i}a_{ij}u_i^{(N)}\na u_j^{(N)}\bigg\|_{L^1(\dom)} \\
	&\le C\|\na u^{(N)}\|_{L^2(\dom)} 
	+ C\|u^{(N)}\|_{L^2(\dom)}\|\na u^{(N)}\|_{L^2(\dom)}.
\end{align*}
Because of the embedding $L^1(\dom)\hookrightarrow H^{m-1}(\dom)'$, it follows that
\begin{align*}
  J_{3} &:= \E\bigg\|\int_{s\wedge t}^{t\lor s}
	\diver\sum_{j=1}^n A_{ij}(u^{(N)}(r))\na u_j^{(N)}(r)\d r\bigg\|_{H^{m}(\dom)'}^2 \\
	&\le C\E\bigg(\int_{s\wedge t}^{t\lor s}
	\bigg\|\sum_{j=1}^n A_{ij}(u^{(N)}(r))\na u_j^{(N)}(r)\bigg\|_{L^1(\dom)}\d r
	\bigg)^2 \\
	&\le C\E\bigg(\int_{s\wedge t}^{t\lor s}\|\na u^{(N)}(r)\|_{L^2(\dom)}\d r\bigg)^2
	+ C\E\bigg(\int_{s\wedge t}^{t\lor s}\|u^{(N)}(r)\|_{L^2(\dom)}
	\|\na u^{(N)}(r)\|_{L^2(\dom)}\d r\bigg)^2 \\
	&=: J_{31} + J_{32}.
\end{align*}
We use the H\"older inequality to obtain
\begin{align*}
  J_{31} &\le C|t-s|\E\int_{s\wedge t}^{t\lor s}
	\|\na u^{(N)}(r)\|_{L^2(\dom)}^2\d r,\\
	J_{32} &\le C\E\bigg\{\bigg(\int_{s\wedge t}^{t\lor s}
	\|u^{(N)}(r)\|^{(2d+4)/d}_{L^2(\dom)}\d r\bigg)^{d/(2d+4)} \\
	&\phantom{xx}{}\times
	\bigg(\int_{s\wedge t}^{t\lor s}\|\na u^{(N)}(r)\|_{L^2(\dom)}^2\d r\bigg)^{1/2}
	|t-s|^{2/(2d+4)}\bigg\}^2 \\
  &= C\E\bigg\{\bigg(\int_{s\wedge t}^{t\lor s}
	\|u^{(N)}(r)\|_{L^2(\dom)}^{(2d+4)/d}\d r\bigg)^{d/(d+2)} 
	\bigg(\int_{s\wedge t}^{t\lor s}\|\na u^{(N)}(r)\|^2_{L^2(\dom)}\d r\bigg)
	|t-s|^{2/(d+2)}\bigg\}.
\end{align*}
This shows that
\begin{align*}
  \int_0^T&\int_0^T|t-s|^{-1-2\alpha}J_{31}\d t\d s
  \le \int_0^T\int_0^T|t-s|^{-2\alpha}\bigg(\E\int_0^T\|\na u^{(N)}(r)\|_{L^2(\dom)}^2
	\d r\bigg)\d t\d s \le C, \\
  \int_0^T&\int_0^T|t-s|^{-1-2\alpha}J_{32}\d t\d s
	\le \int_0^T\int_0^T|t-s|^{-1-2\alpha+2/(d+2)} \\
	&\phantom{xx}\times
	\E\bigg\{\bigg(\int_0^T\|u^{(N)}(r)\|_{L^2(\dom)}^{(2d+4)/d}\d r\bigg)^{d/(d+2)} 
	\bigg(\int_{s\wedge t}^{t\lor s}\|\na u^{(N)}(r)\|^2_{L^2(\dom)}\d r\bigg)\bigg\}
	\d t\d s \\
	&\le \bigg\{\E\bigg(\int_0^T\|u^{(N)}(r)\|_{L^2(\dom)}^{(2d+4)/d}\d r\bigg)^{2d/(d+2)}
	\bigg\}^{1/2} \\
	&\phantom{xx}\times\bigg\{\E\bigg(\int_0^T\int_0^T|t-s|^{-1-2\alpha+2/(d+2)}
	\int_{s\wedge t}^{t\lor s}\|\na u^{(N)}(r)\|^2_{L^2(\dom)}\d r
	\d t\d s\bigg)^2\bigg\}^{1/2} \le C,
\end{align*}
where we used \eqref{2.int} in the last step, requiring that
$1+2\alpha-2/(d+2)<2$ or $\alpha<(d+4)/(2d+4)$. Consequently,
$$
  J_1 \le \int_0^T\int_0^T|t-s|^{-1-2\alpha}(J_{31}+J_{32})\d t\d s \le C.
$$

To estimate $J_2$, we use the embedding $L^2(\dom)\hookrightarrow H^{m}(\dom)'$,
the It\^o isometry, the linear growth of $\sigma$, and the H\"older inequality:
\begin{align*}
  J_2 &\le C\int_0^T\int_0^T|t-s|^{-1-2\alpha}\E\bigg\|\int_{s\wedge t}^{t\lor s}
	\sum_{j=1}^n\sigma_{ij}(u^{(N)}(r))\d W_j(r)\bigg\|_{L^2(\dom)}^2\d t\d s \\
	&= C\int_0^T\int_0^T|t-s|^{-1-2\alpha}\E\bigg(\int_{s\wedge t}^{t\lor s}
  \sum_{j=1}^n\|\sigma_{ij}(u^{(N)}(r))\|_{L^2(\dom)}^2\d r\bigg)\d t\d s \\
	&\le C\int_0^T\int_0^T|t-s|^{-1-2\alpha}\int_{s\wedge t}^{t\lor s}
	\E\sum_{j=1}^n\big(1+\|u_j^{(N)}(r)\|_{L^2(\dom)}^2\big)\d r\d t\d s.
\end{align*}
By \eqref{2.int} and estimate \eqref{2.L3}, 
the right-hand side is finite if $-1-2\alpha>-2$ or $\alpha<1/2$.

{\em Step 2: Case without self-diffusion.}
Let $a_{i0}>0$ and $a_{ii}=0$ for all $i=1,\ldots,n$. 
We estimate the (normally) elliptic operator in $J_1$ 
by exploiting the special structure
of the diffusion term and observing that the embedding $L^1(\dom)\hookrightarrow
H^{m-1}(\dom)'$ is continuous:
\begin{align*}
  \bigg\|\diver\sum_{j=1}^n A_{ij}(u^{(N)})\na u_j^{(N)}\bigg\|_{H^m(\dom)'}
	&= \bigg\|\na\bigg(a_{i0}u_i^{(N)} + \sum_{j\neq i}a_{ij}u_i^{(N)}u_j^{(N)}
	\bigg)\bigg\|_{H^{m-1}(\dom)'} \\
	&\le a_{i0}\|\na u_i^{(N)}\|_{L^1(\dom)} + \sum_{j\neq i}a_{ij}
	\|\na(u_i^{(N)}u_j^{(N)})\|_{L^1(\dom)}.
\end{align*}
This yields, using the definition of $J_3$ from the previous case 
and the H\"older inequality,
\begin{align*}
  J_3 &\le 
	C\E\bigg(\int_{s\wedge t}^{t\lor s}\|\na u_i^{(N)}(r)\|_{L^1(\dom)}\d r\bigg)^2
	+ C\sum_{j\neq i}\E\bigg(\int_{s\wedge t}^{t\lor s}
	\|\na(u_i^{(N)}u_j^{(N)})(r)\|_{L^1(\dom)}\d r\bigg)^2 \\
	&\le C|t-s|\E\int_{s\wedge t}^{t\lor s}\|\na u_i^{(N)}(r)\|_{L^1(\dom)}^2\d r \\
	&\phantom{xx}{}+ C|t-s|^{2/(d+2)}\bigg(\sum_{j\neq i}\E\int_{s\wedge t}^{t\lor s}
	\|\na(u_i^{(N)}u_j^{(N)})(r)\|_{L^1(\dom)}^{\rho_1}\d r\bigg)^{2(d+1)/(d+2)}.
\end{align*}
It follows from \eqref{2.W11.no} and \eqref{2.dW21.no} that
$$
  J_1 \le C\int_0^T\int_0^T\frac{|t-s|^{2/(d+2)}}{|t-s|^{1+2\alpha}}\d s\d t,
$$
and this integral is finite if $2/(d+2)-1-2\alpha>-1$ or
$\alpha<1/(d+2)$.

To estimate $J_2$, we use, similarly as in the case with self-diffusion, 
the embedding $L^2(\dom)\hookrightarrow H^{m}(\dom)'$,
the It\^o isometry, the sublinear growth of $\sigma$, and the H\"older inequality:
\begin{align*}
  J_2 &= C\int_0^T\int_0^T|t-s|^{-1-2\alpha}\E\bigg(\int_{s\wedge t}^{t\lor s}
  \sum_{j=1}^n\|\sigma_{ij}(u^{(N)}(r))\|_{L^2(\dom)}^2\d r\bigg)\d t\d s \\
	&\le C\int_0^T\int_0^T|t-s|^{-1-2\alpha}\int_{s\wedge t}^{t\lor s}
	\E\sum_{j=1}^n\big(1+\|u_j^{(N)}(r)\|_{L^2(\dom)}^{2\gamma}\big)\d r\d t\d s.
\end{align*}
By \eqref{2.int} and estimate \eqref{2.4dL2.no}, the right-hand side is finite if
$2\gamma\le 4/d$ and $-1-2\alpha>-2$ or $\alpha<1/2$.
This finishes the proof.
\end{proof}


\section{Proof of Theorem \ref{thm.ex}}\label{sec.ex}

We prove the existence of a martingale solution in the case with self-diffusion.

\subsection{Tightness of the laws of $(u^{(N)})$}

We show that the laws of $u^{(N)}$ are tight in a certain sub-Polish space.
For this, we introduce the following spaces, recalling that $m>d/2+1$:
\begin{itemize}
\item $C^0([0,T];H^m(\dom)')$ is the space of continuous functions
$u:[0,T]\to H^m(\dom)'$ with the topology $\mathbb{T}_1$ induced by the norm
$\|u\|_{C^0([0,T];H^m(\dom)')}=\sup_{t\in(0,T)}\|u(t)\|_{H^m(\dom)'}$;
\item $L^2_w(0,T;H^1(\dom))$ is the space $L^2(0,T;H^1(\dom))$ with the weak
topology $\mathbb{T}_2$;
\end{itemize}
We define the space
$$
  \widetilde{Z}_T:=C^0([0,T];H^m(\dom)')\cap L_w^2(0,T;H^1(\dom)),
$$
endowed with the topology $\widetilde{\mathbb{T}}$ that is the maximum of the topologies
of $C^0([0,T];H^m(\dom)')$ and $L_w^2(0,T;H^1(\dom))$. Similar to the proof of 
\cite[Lemma 12]{DJZ19}, it can be shown that $\widetilde{Z}_T$ is a sub-Polish space.

\begin{lemma} \label{lem.tight}
The set of laws $({\mathcal L}(u^{(N)}))_{N\in\N}$ is tight in 
$\widetilde{Z}_T$ and in $L^2(0,T;L^2(\dom))$.
\end{lemma}

\begin{proof}
The tightness in $\widetilde{Z}_T$ follows from \cite[Corollary 2.6]{BrMo14}
with the spaces $U=H^m(\dom)$ and $V=H^1(\dom)$. Indeed, estimate \eqref{2.L1}
is exactly condition (a) in \cite[Corollary 2.6]{BrMo14} and estimate \eqref{2.H1}
corresponds to condition (b). Condition (c), i.e., $(u^{(N)})$ satisfies the
Aldous condition in $H^m(\dom)'$, can be verified as in the proof of Lemma 11
in \cite{DJZ19}. Thus, the set of laws of $(u^{(N)})$ is tight in $\widetilde{Z}_T$.

Next, we introduce for given $\kappa>0$ and $\alpha<1/(d+2)$
the sets $X_T=W^{\alpha,2}(0,T;H^m(\dom)')
\cap L^2(0,T;H^1(\dom))$ and $B_\kappa = \{v\in X_T:\|v\|_{X_T}\le\kappa\}$.
The compact embedding $X_T\hookrightarrow L^2(0,T;L^2(\dom))$ \cite[Corollary 5]{Sim87}
implies that $B_\kappa$ is a relatively compact set in $L^2(0,T;$ $L^2(\dom))$.
We deduce from estimates \eqref{2.H1} and \eqref{2.time} and the Chebyshev 
inequality that
\begin{align*}
  \Prob\{\|u^{(N)}\|_{X_T}>\kappa\} &\le \kappa^{-2}\E\|u^{(N)}\|_{X_T}^2 \\
	&\le \kappa^{-2}\big(\E\|u^{(N)}\|_{W^{\alpha,2}(0,T;H^m(\dom)')}^2
	+ \E\|u^{(N)}\|_{L^2(0,T;H^1(\dom))}^2\big) \le C\kappa^{-2}.
\end{align*}
Then the result follows directly from the definition of tightness.
\end{proof}

It follows from the previous lemma that the set of laws $({\mathcal L}(u^{(N)}))$
is tight in $Z_T = \widetilde{Z}_T\cap L^2(0,T;L^2(\dom))$ with the topology
$\mathbb{T}$ that is the maximum of $\widetilde{\mathbb{T}}$ and the topology
induced by the $L^2(0,T;L^2(\dom))$ norm.

\subsection{Strong convergence of $(u^{(N)})$}\label{ssec.strong}

Since $Z_T\times C^0([0,T];\R^n)$ satisfies the assumptions of the
Skorokhod--Jabubowski theorem \cite[Theorem C1]{BrOn10} and the set of laws
$({\mathcal L}(u^{(N)}))$ is tight in $(Z_T,\mathbb{T})$, this theorem implies
the existence of a subsequence of $(u^{(N)})$, which is not relabeled,
a probability space $(\widetilde\Omega,
\widetilde{\mathcal{F}},\widetilde\Prob)$ and, on this space,
$(Z_T\times C^0([0,T];\R^n))$-valued random variables $(\widetilde u,\widetilde W)$
and $(\widetilde u^{(N)},\widetilde W^{(N)})$ for $N\in\N$ such that
$(\widetilde u^{(N)},\widetilde W^{(N)})$ has the same law as $(u^{(N)},W)$
on $\mathcal{B}(Z_T\times C^0([0,T];\R^n))$ and, as $N\to\infty$,
$$
  (\widetilde u^{(N)},\widetilde W^{(N)})\to (\widetilde u,\widetilde W)
	\quad\mbox{in }Z_T\times C^0([0,T];\R^n)\ \widetilde\Prob\mbox{-a.s.}
$$
Because of the definition of the space $Z_T$, this convergence means 
$\widetilde\Prob$-a.s.
\begin{equation}
\label{2.conv}
\begin{aligned}
  \widetilde u^{(N)}\to \widetilde u &\quad\mbox{in }C^0([0,T];H^m(\dom)'),  \\
	\widetilde u^{(N)}\rightharpoonup \widetilde u &\quad\mbox{weakly in }
	L^2(0,T;H^1(\dom)), \\
	\widetilde u^{(N)}\to \widetilde u &\quad\mbox{in }L^2(0,T;L^2(\dom)), \\
  \widetilde W^{(N)}\to \widetilde W &\quad\mbox{in }C^0([0,T];\R^n).
\end{aligned}
\end{equation}

As in \cite{DJZ19}, we derive some regularity properties for the 
limit $\widetilde u$. We infer from the facts that $C^0([0,T];$ $H_N)$ is a 
Borel set of $C^0([0,T];H^m(\dom)')\cap L^2(0,T;L^2(\dom))$, 
$u^{(N)}$ is an element of $C^0([0,T];H_N)$ $\Prob$-a.s., and
$u^{(N)}$ and $\widetilde u^{(N)}$ have the same law on ${\mathcal B}(Z_T)$
that ${\mathcal L}(\widetilde u^{(N)})(C^0([0,T];H_N))=1$. 
Note that $\widetilde u$ is a $Z_T$-Borel random variable since
${\mathcal B}(Z_T\times C^0([0,T];\R^n))$ is a subset of
${\mathcal B}(Z_T)\times{\mathcal B}(C^0([0,T];\R^n))$. We deduce from
estimates \eqref{2.hp} and \eqref{2.nablap} 
and the fact that $u^{(N)}$ and $\widetilde u^{(N)}$
have the same laws that for any $p\ge 2$,
$$
  \sup_{n\in\N}\widetilde{\E}\bigg(\int_0^T\|\widetilde u^{(N)}(t)\|_{H^1(\dom)}^2
	\d t\bigg)^p 
	+ \sup_{n\in\N}\widetilde{\E}\bigg(\sup_{0<t<T}
	\|\widetilde u^{(N)}(t)\|_{H^m(\dom)'}\d t\bigg)^p \le C.
$$
In view of the embedding $L^1(\dom)\hookrightarrow H^m(\dom)'$, we infer the
existence of a subsequence  of $(\widetilde u^{(N)})$ (not relabeled) that is
weakly converging in $L^p(\widetilde\Omega;L^2(0,T;H^1(\dom)))$ and weakly*
converging in $L^p(\widetilde\Omega;L^\infty(0,T;H^m(\dom)'))$ as $N\to\infty$.
In view of the convergence of $\widetilde u^{(N)}\to\widetilde u$ in $Z_T$
$\widetilde\Prob$-a.s., we infer that the limit function satisfies 
$$
  \widetilde\E\bigg(\int_0^T\|\widetilde u(t)\|_{H^1(\dom)}^2\d t\bigg)^p
	+ \widetilde\E\bigg(\sup_{0<t<T}\|\widetilde u(t)\|_{H^m(\dom)'}\bigg)^p < \infty.
$$

Let $\widetilde\F$ and $\widetilde\F^{(N)}$ be the filtrations generated by
$(\widetilde u,\widetilde W)$ and $(\widetilde u^{(N)},\widetilde W^{(N)})$,
respectively. Then we conclude from \cite[Lemma 7]{BGJ13} (also see
\cite[Lemmas 14--15]{DJZ19}) that $\widetilde u$ is progressively measurable
with respect to $\widetilde\F$, $\widetilde u^{(N)}$ is progressively
measurable with respect to $\widetilde\F^{(N)}$, and
$(\widetilde W(t))_{t\in[0,T]}$ and $(\widetilde W^{(N)}(t))_{t\in[0,T]}$
are Wiener processes with respect to the corresponding filtrations.

We know that $u_i^{(N)}$ is nonnegative for $i=1,\ldots,n$. It turns out that
the limit $\widetilde u_i$ is also nonnegative. This is proved in the following lemma.

\begin{lemma}[Nonnegativity]\label{lem.pos}
It holds that $\widetilde u_i(x,t)\ge 0$ for a.e.\ $(x,t)\in Q_T$
$\widetilde\Prob$-a.s.\ and $i=1,\ldots,n$.
\end{lemma}

\begin{proof}
Let $i\in\{1,\ldots,n\}$. The entropy method used in Section \ref{sec.wz} 
implies that $u_i^{(N)}$ is nonnegative, so
$\E\|(u_i^{(N)})^-\|_{L^2(0,T;L^2(\dom))}=0$, where $z^-=\min\{0,z\}$.
The function $u_i^{(N)}$ is $Z_T$-Borel measurable and so does $(u_i^{(N)})^-$.
Therefore, using the equivalence of the laws of $u_i^{(N)}$ and 
$\widetilde u_i^{(N)}$ on $Z_T$ and setting 
$\mu^{(N)}:=\operatorname{Law}(u_i^{(N)})
=\operatorname{Law}(\widetilde u_i^{(N)})$, we find that
\begin{align*}
  \widetilde\E\|(\widetilde u_i^{(N)})^-\|_{L^2(0,T;L^2(\dom))}
	&= \int_{L^2(0,T;L^2(\dom))}\|y^-\|_{L^2(0,T;L^2(\dom))}\d\mu^{(N)}(y) \\
	&= \E\|(u_i^{(N)})^-\|_{L^2(0,T;L^2(\dom))} = 0.
\end{align*}
This shows that $\widetilde u_i^{(N)}\ge 0$ a.e.\ in $Q_T$ 
$\widetilde\Prob$-a.s.
The convergence (up to a subsequence) 
$\widetilde u^{(N)}\to\widetilde u$ a.e.\ in $Z_T$
$\widetilde\Prob$-a.s.\ then implies that $\widetilde u_i\ge 0$ a.e.\ in 
$Q_T$ $\widetilde\Prob$-a.s.
\end{proof}

The following lemma is needed to verify that $(\widetilde u,\widetilde W)$ is
a martingale solution to \eqref{1.eq}--\eqref{1.bic}. In view of the previous
convergence results, the proof is very similar to that one of Lemma 10 in 
\cite{DHJKN20} and therefore, we omit it.

\begin{lemma}\label{lem.E}
It holds for all $s$, $t\in[0,T]$ with $s\le t$, $i=1,\ldots,n$, and all
$\phi_1\in L^2(\dom)$ and all $\phi_2\in H^m(\dom)$ with $m>d/2+1$ and
satisfying $\na\phi_2\cdot\nu=0$ on $\pa\dom$ that
\begin{align*}
  \lim_{N\to\infty}\widetilde\E\int_0^T\big\langle\widetilde u_{i}^{(N)}(t)
	-\widetilde u_{i}(t),
	\phi_2\big\rangle\d t &= 0, \\ 
	\lim_{N\to\infty}\widetilde\E\big\langle\widetilde u_{i}^{(N)}(0)
	-\widetilde u_{i}(0),\phi_2\big\rangle &= 0, \\ 
	\lim_{N\to\infty}\widetilde\E\int_0^T\bigg|\sum_{j=1}^n\int_0^t\Big\langle
	A_{ij}(\widetilde u^{(N)}(s))\na\widetilde u_j^{(N)}(s)
	- A_{ij}(\widetilde u(s))\na\widetilde u_j(s),\na\phi_2\Big\rangle
	\d s\bigg|\d t &= 0, \\ 
	\lim_{N\to\infty}\widetilde\E\int_0^T\bigg|\sum_{j=1}^n\int_0^t
	\Big\langle\sigma_{ij}(\widetilde u^{(N)}(s))\textnormal{d}\widetilde W_j^{(N)}(s)
	-\sigma_{ij}(\widetilde u(s))\d\widetilde W_j(s),\phi_1\Big\rangle\bigg|^2 
	\d t &= 0. 
\end{align*}
\end{lemma}

Next, we define for $t\in[0,T]$ and $i=1,\ldots,n$,
\begin{align}
  \Lambda^{(N)}_i\big(\widetilde u^{(N)},\widetilde W^{(N)},\phi\big)(t)
	&:= \big\langle\Pi_N(\widetilde u_i(0)),\phi\big\rangle \nonumber \\
	&\phantom{xx}{}
	- \int_0^t\sum_{j=1}^n\big\langle \Pi_N(A_{ij}(\widetilde u^{(N)}(s))
	\na\widetilde u^{(N)}_j(s)),\na\phi\big\rangle \d s \nonumber \\
	&\phantom{xx}{}
	+ \sum_{j=1}^n\bigg\langle\int_0^t\Pi_N\sigma_{ij}(\widetilde u^{(N)}(s))
	\textnormal{d}\widetilde W_j^{(N)}(s),\phi\bigg\rangle, \label{2.LamN} \\
	\Lambda_i\big(\widetilde u,\widetilde W,\phi\big)(t)
	&:= \langle\widetilde u_i(0),\phi\rangle
	- \sum_{j=1}^n\int_0^t\big\langle A_{ij}(\widetilde u(s))
	\na\widetilde u_j(s),\na\phi\big\rangle \d s \nonumber \\ 
	&\phantom{xx}{}+ \sum_{j=1}^n\bigg\langle\int_0^t\sigma_{ij}(\widetilde u(s))
	\d\widetilde W_j(s),\phi\bigg\rangle. \label{2.Lam}
\end{align}
The following corollary is essentially 
a consequence of Lemma \ref{lem.E}; see \cite[Corollary 17]{DJZ19} for a proof.

\begin{corollary}\label{coro.E}
It holds for any $\phi_1\in L^2(\dom)$ and any $\phi_2\in H^m(\dom)$ 
with $m>d/2+1$ and satisfying $\na\phi_2\cdot\nu=0$ on $\pa\dom$ that
\begin{align*}
  \lim_{N\to\infty}\big\|\langle\widetilde u_{i}^{(N)},\phi_2\rangle
	- \langle\widetilde u_{i},\phi_2\rangle\big\|_{L^1(\widetilde\Omega\times(0,T))} 
	&= 0, \\
  \lim_{N\to\infty}\big\|\Lambda_i^{(N)}\big(\widetilde u^{(N)},\widetilde W^{(N)},
	\phi_1\big)	- \Lambda_i\big(\widetilde u,\widetilde W,\phi_1\big)
	\big\|_{L^1(\widetilde\Omega\times(0,T))} 
	&= 0.
\end{align*}
\end{corollary}

We proceed with the proof of Theorem \ref{thm.ex}. Since $\widetilde u^{(N)}$
is a strong solution to \eqref{2.sde}--\eqref{2.ic}, it satisfies 
$$
  \langle u^{(N)}_i(t),\phi\rangle = \Lambda_i^{(N)}(u^{(N)},W,\phi)(t)
	\quad\mbox{for a.e. }t\in[0,T]\ \Prob\mbox{-a.s.},\ i=1,\ldots,n,
$$
for any $\phi\in H^m(\dom)$. Hence,
$$
  \int_0^T\E\big|\langle u_i^{(N)}(t),\phi\rangle - \Lambda_i^{(N)}(u^{(N)},W,\phi)(t)
	\big|\d t = 0, \quad i=1,\ldots,n.
$$
We deduce from the equivalence of the laws of $(u^{(N)},W)$ and 
$(\widetilde u^{(N)},\widetilde W^{(N)})$ that
$$
  \int_0^T\widetilde\E\big|\langle\widetilde u_i^{(N)}(t),\phi\rangle
	- \Lambda_i^{(N)}(\widetilde u^{(N)},\widetilde W^{(N)},\phi)(t)\big|\d t = 0,
	\quad i=1,\ldots,n.
$$
By Corollary \ref{coro.E}, we can pass to the limit $N\to\infty$ to obtain
$$
  \int_0^T\widetilde\E\big|\langle\widetilde u_i(t),\phi\rangle
	- \Lambda_i(\widetilde u,\widetilde W,\phi)(t)\big|\d t = 0, \quad i=1,\ldots,n.
$$
This identity holds for all $\phi\in H^m(\dom)$ such that $\na\phi\cdot\nu=0$ on
$\pa\dom$ and hence, by density, for any $\phi\in H^m(\dom)$. This shows that
$$
  \big|\langle\widetilde u_i(t),\phi\rangle
	- \Lambda_i(\widetilde u,\widetilde W,\phi)(t)\big| = 0
	\quad\mbox{for a.e. }t\in[0,T]\ \widetilde\Prob\mbox{-a.s.},\ i=1,\ldots,n.
$$
We infer from the definition of $\Lambda_i$ that
$$
  \langle \widetilde u_i(t),\phi\rangle 
	= \langle\widetilde u_i(0),\phi\rangle 
	- \sum_{j=1}^n\int_0^t\big\langle A_{ij}(\widetilde u(s))
	\na\widetilde u_j(s),\na\phi\big\rangle \d s 
	+ \sum_{j=1}^n\bigg\langle\int_0^t\sigma_{ij}(\widetilde u(s))
	\d\widetilde W_j(s),\phi\bigg\rangle
$$
for a.e.\ $t\in[0,T]$ $\widetilde\Prob$-a.s.\ and all $\phi\in H^m(\dom)$. Set
$\widetilde U=(\widetilde\Omega,\widetilde\F,\widetilde\F,\widetilde\Prob)$.
Then $(\widetilde U,\widetilde u,\widetilde W)$ is a martingale solution to
\eqref{1.eq}--\eqref{1.bic}. This finishes the proof.


\section{Proof of Theorem \ref{thm.no}}\label{sec.no}

Next, we prove the existence of a martingale solution in the case without 
self-diffusion. 
The proof is similar to that one of Theorem \ref{thm.ex}, but we have less regularity
for $u_i^{(N)}$ than in the self-diffusion case. Therefore, we need to adapt
the function spaces. Moreover, since we do not have a uniform estimate for $u_i^{(N)}$
in $L^2(Q_T)$, the convergence of $\na(u_i^{(N)}u_j^{(N)})$
requires some care. Recall that we assume $d\le 3$.

\subsection{Tightness of the laws}

We show that the laws of $u^{(N)}$ are tight in the sub-Polish space
$$
  \widetilde Z_T := C^0([0,T];H^3(\dom)')\cap L_w^{8/7}(0,T;W^{1,8/7}(\dom)),
$$
endowed with the topology $\widetilde{\mathbb T}$
with is the maximum of the topology of
$C^0([0,T];H^3(\dom)')$ and the weak topology of $L_w^{8/7}(0,T;W^{1,8/7}(\dom))$.

\begin{lemma}
The sequence of laws $(\L(u^{(N)}))_{N\in\N}$ is tight in $\widetilde Z_T$,
$L^{5/4}(0,T;L^2(\dom))$, and $L^2(0,T;L^{3/2}(\dom))$. 
\end{lemma}

\begin{proof}
The tightness in $\widetilde Z_T$ can be shown as in Lemma \ref{lem.tight}.
Furthermore, the tightness in $L^{5/4}(0,T;L^2(\dom))$ follows similarly
as in Lemma \ref{lem.tight}, observing that the embedding $W^{\alpha,2}(0,T;H^3(\dom)')
\cap L^{5/4}(0,T;W^{1,5/4}(\dom))\hookrightarrow L^{5/4}(0,T;L^2(\dom))$ is compact.
Here, we use the facts that the embedding $W^{1,5/4}(\dom)\hookrightarrow L^2(\dom)$ 
is compact if $d\le 3$ and that $(u_i^{(N)})$ is bounded in
$W^{\alpha,2}(0,T;H^3(\dom)')$ and $L^{5/4}(0,T;W^{1,5/4}(\dom))$ due to 
estimates \eqref{2.time} and \eqref{2.W1r.no}, respectively. 
Finally, the last statement is a consequence of the
compact embedding $W^{\alpha,2}(0,T;H^3(\dom)')\cap L^2(0,T;W^{1,1}(\dom))
\hookrightarrow L^2(0,T;L^p(\dom))$ for any $p<3/2$ 
as well as estimates \eqref{2.time} and
\eqref{2.W11.no}. In fact, the compactness in $L^2(0,T;L^p(\dom))$
is valid up to $p=3/2$ by taking into account the uniform bound of
$u_i^{(N)}\log u_i^{(N)}$ in $L^\infty(0,T;L^1(\dom))$; see \cite[Prop.~1]{BCJ20}.
\end{proof}

The previous lemma shows that $(\L(u^{(N)}))$ is tight in
$$
  Z_T = \widetilde Z_T\cap L^2(0,T;L^{3/2}(\dom))\cap L^{5/4}(0,T;L^2(\dom))
$$
with the topology that is the maximum of $\widetilde{\mathbb T}$ and the
topologies induced by the $L^2(0,T;$ $L^{3/2}(\dom))$ and 
$L^{5/4}(0,T;L^2(\dom))$ norms.

\subsection{Strong convergence of $(u^{(N)})$}

Applying the Shorokhod--Jabubowski theorem as in Section \ref{ssec.strong},
we obtain the existence of a subsequence of $(u^{(N)})$ (not relabeled),
a probability space $(\widetilde\Omega,\widetilde\F,\widetilde\Prob)$ and,
on this space, $(Z_T\times C^0([0,T];\R^n))$-valued random variables 
$(\widetilde u,\widetilde W)$ and $(\widetilde u^{(N)},\widetilde W^{(N)})$
for $N\in\N$ such that $(\widetilde u^{(N)},\widetilde W^{(N)})$ has the same
laws as $(u^{(N)},W)$ on $\mathcal{B}(Z_T\times C^0([0,T];\R^n))$.
The convergence results \eqref{2.conv} hold with the exception that only
\begin{equation}\label{3.conv}
\begin{aligned}
	\widetilde u^{(N)}\rightharpoonup \widetilde u &\quad\mbox{weakly in }
	L^{8/7}(0,T;W^{1,8/7}(\dom)), \\
	\widetilde u^{(N)}\to \widetilde u &\quad\mbox{in }L^2(0,T;L^{3/2}(\dom)), \\
	\widetilde u^{(N)}\to \widetilde u &\quad\mbox{in }L^{5/4}(0,T;L^2(\dom))
\end{aligned}
\end{equation}
$\widetilde\Prob$-a.s.\ as $N\to\infty$. 
Moreover, similarly as in Section \ref{ssec.strong}, we deduce from
estimates \eqref{2.W1r.no} and \eqref{2.hp} that
$$
  \sup_{n\in\N}\widetilde{\E}\bigg(\int_0^T
	\|\widetilde u^{(N)}(t)\|_{W^{1,8/7}(\dom)}^{8/7}\d t\bigg)^p 
	+ \sup_{n\in\N}\widetilde{\E}\bigg(\sup_{0<t<T}\|\widetilde u^{(N)}(t)\|_{H^m(\dom)'}
	\d t\bigg)^p \le C.
$$
Because of the measurability of the map
$$
  C^0([0,T];H_N)\to L^{8/7}(0,T;W^{1,8/7}(\dom))\cap Z_T, w\mapsto w^2,
$$
the continuity of the norm on $L^{8/7}(0,T;W^{1,8/7}(\dom))\cap Z_T$, the
identity $u_i^{(N)}u_j^{(N)} = (u_i^{(N)}+u_j^{(N)})^2 - (u_i^{(N)})^2
- (u_j^{(N)})^2$, the equality of the laws of $u^{(N)}$ and $\widetilde u^{(N)}$
on $\mathcal{B}(Z_T)$, and estimates \eqref{2.dL24.no} and \eqref{2.dnabla.no} 
for $d=3$, we obtain
\begin{equation}\label{3.W1.no}
  \sup_{N \in \N} \widetilde\E \bigg(\int_0^T\|\widetilde{u}_i^{(N)}
	\widetilde{u}_j^{(N)}\|_{W^{1,8/7}(\dom)}^{8/7} \d t\bigg)^p 
	\le C.
\end{equation}
We infer the existence of a subsequence of $(\widetilde u^{(N)})$ (not relabeled) 
that is weakly converging in $L^p(\widetilde\Omega;L^{8/7}(0,T;W^{1,8/7}(\dom)))$ 
and weakly* converging in $L^p(\widetilde\Omega;C^0([0,T];H^3(\dom)'))$ 
as $N\to\infty$.
Thus, taking into account the convergence $\widetilde u^{(N)}\to\widetilde u$ in $Z_T$
$\widetilde\Prob$-a.s., we conclude that the limit function satisfies
$$
  \widetilde\E\bigg(\int_0^T\|\widetilde u(t)\|_{W^{1,8/7}(\dom)}^{8/7}\d t\bigg)^p
	+ \widetilde\E\bigg(\sup_{0<t<T}\|\widetilde u(t)\|_{H^3(\dom)'}\bigg)^p < \infty
	\quad\mbox{for }p<\infty.
$$

We verify similarly as in Lemma \ref{lem.pos} that $\widetilde u_i(x,t)\ge 0$ for
a.e.\ $(x,t)\in Q_T$ $\widetilde\Prob$-a.s.\ and $i=1,\ldots,n$.
The only difference to the proof is that we work in the space $L^{5/4}(0,T;L^2(\dom))$
instead of $L^2(0,T;L^2(\dom))$. 

The following lemma allows us to identify the quadratic terms.

\begin{lemma}\label{lem.uiuj}
Let $(\widetilde u^{(N)})$ be the sequence of $Z_T$-valued random variables
constructed above. Then it holds for $\phi_2\in W^{1,8}(\dom)$
and satisfying $\na\phi_2\cdot\nu=0$ on $\pa\dom$ that
$$
  \lim_{N\to\infty}\widetilde{\E}\int_0^T\bigg|\sum_{i,j=1,\,j\neq i}^n\int_0^t
	\big\langle\na(\widetilde u_i^{(N)}\widetilde u_j^{(N)})(s)
	- \na(\widetilde u_i\widetilde u_j)(s),\na\phi_2\big\rangle\d s\bigg|\d t = 0.
$$
\end{lemma}

\begin{proof}
We infer from estimate \eqref{3.conv} that, for a subsequence,
$\widetilde u_i^{(N)}(x,t)\to \widetilde u_i(x,t)$ and also
$(\widetilde u_i^{(N)}\widetilde u_j^{(N)})(x,t)\to 
(\widetilde u_i\widetilde u_j)(x,t)$ for a.e.\ $(x,t)\in Q_T$
$\widetilde\Prob$-a.s. Then $f^{(N)}:=\widetilde u_i^{(N)}\widetilde u_j^{(N)}
-\widetilde u_i\widetilde u_j\to 0$ a.e.\ in $Q_T$ $\widetilde\Prob$-a.s.
Taking into account the uniform bound for $f^{(N)}$ in
$L^p(\widetilde\Omega;L^{8/7}(0,T;$ $W^{1,8/7}(\dom)))$ from \eqref{3.W1.no},
we conclude the strong convergence $f^{(N)}\to 0$ in $L^p(\widetilde\Omega\times
\dom\times(0,T))$ for any $1\le p<8/7$.
Next, let $g^{(N)}:=f^{(N)}\Delta\phi$ for $\phi\in W^{2,\infty}(\dom)$.
Then $g^{(N)}\to 0$ in $Q_T$ $\widetilde\Prob$-a.s.\ and $g^{(N)}$
is bounded in $L^{8/7}(\widetilde\Omega\times\dom\times(0,T))$. Therefore,
$g^{(N)}\to 0$ strongly in $L^p(\widetilde\Omega\times
\dom\times(0,T))$ for any $1<p<8/7$. This shows that
$$
  \lim_{N\to\infty}\widetilde{\E}\sum_{i,j=1,\,j\neq i}^n\int_0^t
	\bigg|\big\langle(\widetilde u_i^{(N)}\widetilde u_j^{(N)})(s)
	- (\widetilde u_i\widetilde u_j)(s),\Delta\phi\big\rangle\bigg|\d s = 0
$$
for any $\phi\in W^{2,\infty}(\dom)$ and, by integrating by parts,
\begin{align*}
   \lim_{N\to\infty}\widetilde{\E}&\int_0^T\bigg|\sum_{i,j=1,\,j\neq i}^n\int_0^t
	\big\langle\na(\widetilde u_i^{(N)}\widetilde u_j^{(N)})(s)
	- \na(\widetilde u_i\widetilde u_j)(s),\na\phi\big\rangle\d s\bigg|\d t\\
	&\le \lim_{N\to\infty}T\widetilde{\E}\sum_{i,j=1,\,j\neq i}^n\int_0^T
	\bigg|\big\langle(\widetilde u_i^{(N)}\widetilde u_j^{(N)})(s)
	- (\widetilde u_i\widetilde u_j)(s),\Delta\phi\big\rangle\bigg|\d s = 0
\end{align*}
for any $\phi\in W^{2,\infty}(\dom)$ satisfying $\na\phi\cdot\nu=0$ on $\pa\dom$.
By density, the convergence also holds for $\phi\in W^{1,8}(\dom)$ satisfying
$\na\phi\cdot\nu=0$ on $\pa\dom$.
\end{proof}

Lemma \ref{lem.E}, which is needed to show
that $(\widetilde u,\widetilde W)$ is a martingale solution to 
\eqref{1.eq}--\eqref{1.bic}, also holds in the present situation.
The proof is similar to that one of Lemma 10
in \cite{DHJKN20} and uses the previous convergence results, convergences
\eqref{3.conv}, Lemma \ref{lem.uiuj}, and the sublinear growth of the
multiplicative noise.

Defining $\Lambda_i^{(N)}$ and $\Lambda_i$ as in \eqref{2.LamN} and \eqref{2.Lam},
respectively, the same result as in Corollary \ref{coro.E} holds. The proof of
Theorem \ref{thm.no} can now be finished as in Section \ref{ssec.strong}.


\begin{appendix}
\section{Deterministic SKT system without self-diffusion}\label{app}

The proof of the existence of a global weak solution to the deterministic two-species 
SKT system without self-diffusion
in \cite{ChJu06} uses an $L^2\log L^2$ bound coming from the Lotka--Volterra
terms. We claim that the proof can be performed without
this bound. To show this claim, we recall the estimates coming from the
entropy inequality proved in \cite{ChJu06}:
$$
  \|u_i^{(\tau)}\|_{L^\infty(0,T;L^1(\dom))}
	+ \|\na(u_i^{(\tau)})^{1/2}\|_{L^2(Q_T)}
	+ \|\na(u_i^{(\tau)}u_j^{(\tau)})^{1/2}\|_{L^2(Q_T)} \le C,
$$
where $i\neq j$ and $C>0$ does not depend on the approximation
parameter $\tau$. In particular, $(u_i^{(\tau)})$ is bounded in $L^2(0,T;H^1(\dom))$. 
The function $u_i^{(\tau)}$ is the solution to an approximate problem, which
we do not specify here; we refer to \cite{ChJu06}.

First, we show that $(u_i^{(\tau)})$ is bounded in 
$L^{\rho_1}(0,T;W^{1,\rho_1}(\dom))$, where
$\rho_1=(d+2)/(d+1)$. To this end, we deduce from the Gagliardo--Nirenberg inequality 
with $\theta=d/2-d/p$ and $p=2+4/d$ (satisfying $p\theta=2$) that
\begin{align*}
  \|(u_i^{(\tau)})^{1/2}\|_{L^{2+4/d}(Q_T)}^{2+4/d}
	&\le C\int_0^T\|(u_i^{(\tau)})^{1/2}\|_{H^1(\dom)}^{\theta(2d+4)/d}
	\|(u_i^{(\tau)})^{1/2}\|_{L^2(\dom)}^{(1-\theta)(2d+4)/d}\d t \\
	&\le \|u_i^{(\tau)}\|_{L^\infty(0,T;L^1(\dom))}^{(1-\theta)(d+2)/d}\int_0^T
	\|(u_i^{(\tau)})^{1/2}\|_{H^1(\dom)}^2 \d t \le C.
\end{align*}
This bound and the $L^2(Q_T)$ bound for $\na(u_i^{(\tau)})^{1/2}$ show that
$$
  \|\na u_i^{(\tau)}\|_{L^{\rho_1}(Q_T)} = 2\|(u_i^{(\tau)})^{1/2}\|_{L^{2+4/d}(Q_T)}
	\|\na(u_i^{(\tau)})^{1/2}\|_{L^2(Q_T)} \le C.
$$
The claim now follows from the Poincar\'e--Wirtinger inequality and the bound
for $(u_i^{(\tau)})$ in $L^\infty(0,T;L^1(\dom))$. 

Second, we claim that $(\na(u_i^{(\tau)}u_j^{(\tau)})^{1/2})$
is bounded in $L^{\rho_2}(Q_T)$ for $i\neq j$, where $\rho_2=(2d+2)/(2d+1)$. Indeed,
a similar argument as above, using the bounds for $(u_i^{(\tau)}u_j^{(\tau)})^{1/2}$
in $L^\infty(0,T;L^1(\dom))$ and $\na(u_i^{(\tau)}u_j^{(\tau)})^{1/2}$
in $L^2(Q_T)$ as well as the Gagliardo--Nirenberg inequality, shows that 
$((u_i^{(\tau)}u_j^{(\tau)})^{1/2})$ is bounded in $L^{2+2/d}(Q_T)$. 
Therefore, the sequence
$$
  \na(u_i^{(\tau)}u_j^{(\tau)})
	= 2(u_i^{(\tau)}u_j^{(\tau)})^{1/2}\na(u_i^{(\tau)}u_j^{(\tau)})^{1/2}
$$
is bounded in $L^{\rho_2}(Q_T)$, proving the claim.

For the compactness, we also need an estimate for the time derivative:
$$
  \|\pa_t u_i^{(\tau)}\|_{L^{\rho_1}(0,T;W^{1,\rho_1}(\dom)')}
	\le \bigg\|a_{i0}\na u_i^{(\tau)} + \sum_{j=1,\, j\neq i}^n a_{ij}
	\na(u_i^{(\tau)}u_j^{(\tau)})\bigg\|_{L^{\rho_2}(Q_T)} \le C.
$$
(In fact, in the proof of \cite{ChJu06}, we have to replace $\pa_t u_i$
by a discrete time derivative, but this does not change the argument.)
By the Aubin--Lions lemma, there exists a subsequence of $(u_i^{(\tau)})$
(not relabeled) such that $u_i^{(\tau)}\to u_i$ strongly in $L^{\rho_1}(Q_T)$
and a.e.\ in $Q_T$ as $\tau\to 0$. Thus, $u_i^{(\tau)}u_j^{(\tau)}\to u_iu_j$
a.e.\ in $Q_T$. The bound for $(\na(u_i^{(\tau)}u_j^{(\tau)}))$ in $L^{\rho_2}(Q_T)$ 
for $i\neq j$ implies that $\na(u_i^{(\tau)}u_j^{(\tau)})
\rightharpoonup\na(u_iu_j)$ weakly in $L^{\rho_2}(Q_T)$. 
Furthermore, we have the convergences
$\na u_i^{(\tau)}\rightharpoonup \na u_i$ weakly in $L^{\rho_1}(Q_T)$ and 
$\pa_t u_i^{(\tau)}\rightharpoonup\pa_t u_i$ weakly in 
$L^{\rho_2}(0,T;W^{1,\rho_2}(\dom)')$.
These limits allow us to pass to the limit $\tau\to 0$ in the approximate
problem. Moreover, we obtain the regularity results formulated
in Section \ref{sec.wz}. As a corollary, we deduce the following existence
result which extends \cite[Theorem 1]{ChJu06} to the $n$-species no-reaction case.

\begin{theorem}[Existence for the deterministic system]
Let $\Omega\subset\R^d$ ($d\ge 1$), $u^0\in L^\infty(\dom;\R^n)$ with $u_i\ge 0$
a.e.\ in $\dom$, let the detailed-balance condition \eqref{2.dbc} hold, and let 
$a_{i0}>0$, $a_{ii}=0$ for $i=1,\ldots,n$. Then there exists a weak solution 
$u=(u_1,\ldots,u_n)$ to
\begin{align*}
  & \pa_t u_i = \Delta\bigg(a_{i0}u_i + \sum_{j=1,\,j\neq i}^n a_{ij}u_iu_j\bigg)
	\quad\mbox{in }\dom,\ t>0,\ i=1,\ldots,n, \\
  & u_i(0)=u_i^0\quad\mbox{in }\dom, \quad\na u_i\cdot\nu=0
	\quad\mbox{on }\pa\dom,\ t>0,
\end{align*}
satisfying $u_i(t)\ge 0$ a.e.\ in $\dom$, $t>0$ and
$$
  u_i\in L^{\rho_1}(0,T;W^{1,\rho_1}(\dom)), 
	\quad u_iu_j\in L^{\rho_2}(0,T;W^{1,\rho_2}(\dom)),
	\quad \pa_t u_i\in L^{\rho_2}(0,T;W^{1,\rho_2}(\dom)')
$$
for $i=1,\ldots,n$, where $\rho_1=(d+2)/(d+1)$ and $\rho_2=(2d+2)/(2d+1)$.
\end{theorem}

\end{appendix}


\end{document}